\documentclass[a4paper,10pt]{article}
\usepackage[utf8]{inputenc}
\usepackage[T1]{fontenc}
\usepackage{lmodern}
\usepackage{amsfonts}
\usepackage{a4wide}
\usepackage{geometry}
\usepackage{graphicx}
\usepackage{amssymb, amsmath, amsopn, amstext, amscd, amsthm}
\usepackage{color}
\usepackage{version}
\usepackage{todonotes}
\usepackage{xcolor}

\newtheorem{definition}{Definition}
\newtheorem{proposition}{Proposition}
\newtheorem{theorem}{Theorem}
\newtheorem{lm}{Lemma}
\newtheorem{remark}{Remark}

\newcommand{\eps}{\varepsilon}

\newcommand{\1}{{\bf 1}}
\newcommand{\card}{{\rm card}}
\newcommand{\Xp}{X_{\bullet}}
\newcommand{\Yp}{Y_{\bullet}}
\newcommand{\Pro}{\mathbb{P}}

\begin{document}

\title{Consistency of plug-in confidence sets for \\
classification in semi-supervised learning}
 \author{Christophe Denis\footnote{Christophe.Denis@u-pem.fr}  \, and Mohamed Hebiri\footnote{Mohamed.Hebiri@u-pem.fr}\\
\small{ LAMA, UMR-CNRS 8050,}\\
\small{Universit\'{e} Paris Est -- Marne-la-Vall\'ee}
}
 \date{}
\maketitle

\begin{abstract}
Confident prediction is highly relevant in machine learning; for example,
in applications such as medical diagnoses, wrong prediction can be fatal.
For classification, there already exist procedures that allow to not 
classify data when the confidence in their prediction is weak. 
This approach is known as classification with reject option.
In the present paper, we provide new methodology for this approach.
Predicting a new instance via a confidence set, 
we ensure an exact control of the probability of classification.
Moreover, we show that this methodology is easily implementable
and entails attractive theoretical and numerical properties. 
 \vspace*{0.25cm} 
\hspace*{1.60cm}  confidence sets,  plug-in confidence sets.
\end{abstract}

\section{Introduction}
\label{Sec:Intro}

Binary classification aims at assigning a label $Y \in \{0,1\}$ to a given example $X \in \mathcal{X}$.
The goal is then to build a classification rule $s : \mathcal{X} \to \{0,1\}$ so that $s(X)$, the predicted label
for the observed example $X$ is as close as possible to the label $Y$.
In this framework the question of confident prediction, which results in wondering 
how accurate is the prediction $s(X)$, becomes a central question.
Doubts about the confidence of the predicted label $s(X)$ may arise in these situations: if
the conditional probability $\eta^* (x) = \mathbb{P}(Y=1 | X=x)$ is close to $1/2$ so that the feature $x$ might be hard to classify
whatever the classification rule is; or, if the classification rule is inefficient.
In such cases, it is worth considering procedures that allow to not classify an observation when the doubt is too important. 
We talk about 
{\it classification with reject option}. This setting is particularly relevant in some applications where wrong classification may lead to big issues:
it is hence better to not assign a label rather than to assign a non confident one. 
Procedures for classification with reject option has been studied by several authors~\cite{Ch70,NZH10,GRAN09CI,Vovk99IntroCP,Vovk_livre,HW06,BW08,WY11,LeiConfTheori}
and references therein. In this context, two questions arise:
how to determine whether we should classify an example or not; and how to take into account the reject option?
The works on classification with reject option can be separated according to two approaches.

i) The works which rely on the {\it conformal predictors} algorithm~\cite{Vovk99IntroCP,Vovk_livre}.
The general idea of conformal prediction is to build for a given feature $X$ a set $\Gamma(X)$, 
which takes its value in $\{ \emptyset, \{0\},\{1\},\{0,1\}\}$,
and contains the true label with high probability. 
The feature $X$ is not classified if $\card(\Gamma(X)) \neq 1$.
One of the most important ideas behind the construction of conformal predictors is the notion of conformity.
More precisely, the value of $\Gamma(X)$ depends on the similarity between the example $X$ and an already collected labeled dataset.
Then, the procedure uses local arguments and can be seen as a transductive method \cite{localEstimation,VapnikLivre}.
In terms of performance, the set $\Gamma(X)$ is built in order to control the overall misclassification risk:
for a given significance level $\varepsilon \in (0,1)$, the set $\Gamma(X)$ satisfies\footnote{In the terminology of conformal predictors, both of the outputs $\emptyset$ and $ \{0,1\}$ mean that no label is assigned. Both are important to be able to guarantee the exact control of the overall misclassification risk regardless of the classification rule used to build the set $\Gamma(X)$.}
$\mathbb{P}\left( Y  \notin \Gamma(X)  \right) \leq \eps$.
The major drawback of the conformal prediction approach is that it does not take into account the reject option in the risk. 
Moreover, if the significance level $\eps$ is too small, the
resulting set $\Gamma(X)$ belongs to $\{\emptyset, \{0,1\}\}$ for all $X$.
Hence, the use of reject option is irrelevant.

ii) The other works rely on the setting provided in~\cite{Ch70,HW06,BW08,WY11}. In this case, a {\it classification rule with reject option} 
$s_R$ takes its values in $\{0,1,R_e\}$, where $s_R(X) = R_e$ means reject: no label is affected to the instance $X$. 
In the above mentioned works, rejecting is viewed as an error and for a fixed value of some parameter
$\alpha \in [1/2,1]$, the cost of the rejection is $1- \alpha$. Therefore, the risk function associated
with a classifier with reject option is given by
$L_{\alpha}(s_R) = \mathbb{P}\left(\{{s}_R(X) \neq Y \}  \  \text{and} \  \{X \text{ is classified} \} \right) + (1-\alpha)   \mathbb{P}\left(\{X \text{ is rejected} 
\} \right)$.
The results provided in~\cite{Ch70} illustrate that the optimal reject procedure for $L_{\alpha}$ is given by 
\begin{equation*}
s^{*}_{R_{\alpha}}(X) = 		\begin{cases}
			0 \;\; \qquad {\rm if} \; \eta^*(X) \leq 1- \alpha, \\
			1 \;\; \qquad {\rm if} \; \eta^*(X) \geq \alpha, \\
			R_e \qquad {\rm otherwise}.
		\end{cases} 
\end{equation*}
Herbei and Wegkamp~\cite{HW06} study the asymptotic optimality of procedures based on plug-in rules or on empirical risk minimization.
We address some limits of this approach. First, the choice of the parameter $\alpha$ is fundamental for the procedure
and fixing it is tricky. As an immediate consequence, if the value of parameter $\alpha$ is either too small or too large, the use of reject option
can be irrelevant. Moreover, this approach does not allow to control any of the two parts of the risk function, in particular the rejection probability. 
Hence, comparing two classifiers with reject option in terms of the risk function $L_{\alpha}$ remains difficult to interpret: they do not have necessarily
the same rejection probabilities.

Both approaches previously presented bring into play the reject option through rather a set (conformal predictor) or a classifier with reject option. However, none provides a control on the probability of classifying a feature.
In the present paper we consider a new way to tackle the problem of classification with reject option.
We aim at controlling the rejection probability and at bounding the misclassification risk restricted to the set of label examples. Both considerations are new.
For a given classifier $s$ and a feature $X$, our methodology involves a statistical procedure which provides 
a set $\Gamma_s(X) \in \{\{0\},\{1\},\{0,1\}\}$, namely a {\it confidence set}.
We introduce in the present work oracle confidence sets, says $\Gamma^*$, which relies on a score function deduced from $\eta^*$ and also on its cumulative distribution function. 
The main characteristic of oracle confidence sets is that they are able to control exactly the rejection probability $\mathbb{P}(\Gamma^*(X)  = \{0,1\})$: under mild assumption, we get level-$\eps$-confidence sets.
These sets are called $\eps$-confidence sets. This aspect makes our procedure prevent irrelevant use of reject option.
Hence we do not view the reject as an error but simply as a parameter that we are able to control; 
moreover, we evaluate the quality of a confidence set through the misclassification risk conditionally on the set of classified
examples. That is,
for a given classification rule $s$, we focus on the control of the risk function  
$\mathcal{R}(\Gamma_s) = \mathbb{P}\left(Y \in \Gamma_s(X)  \ | \  \{X \text{ is classified} \} \right)$.
To the best of our knowledge, none of the earlier works provides a control of this risk neither of the rejection probability.
According to the risk function $\mathcal{R}$, for $\eps \in ]0,1]$, the $\eps$-confidence sets are shown to be optimal 
over the set of all confidence sets with rejection probability equal to $1- \eps$.
Another contribution of the paper is to provide an algorithm which involves a consistent estimator of $\eta^*$ and yields a confidence set.
For a given level $\eps\in ]0,1]$, we do not build only a single algorithm of constructing asymptotically level-$\varepsilon$-confidence sets, but a general device
that takes as input a consistent estimator of the regression function and a unlabeled sample, and produces as output a confidence set which is provably asymptotically of level $\varepsilon$ and consistent (i.e., the excess risk tends to zero).
The resulting confidence set is referred as plug-in $\eps$-confidence set.
Furthermore, we establish rates of convergence under the Tsybakov noise assumption on the data generating distribution.
Moreover, these confidence sets have the advantage of being easily implementable.

The rest of the paper is organized as follows.
The definition and the important properties of the $\eps$-confidence sets are provided in Section~\ref{sec:ourConfSet}.
We also apply the $\eps$-confidence sets in the Gaussian mixture model.
Section~\ref{sec:Estimateion} is devoted to the introduction of the plug-in $\eps$-confidence sets and their asymptotic behavior.
We present a numerical illustration of our results in Section~\ref{sec:simu}. We finally draw some conclusions and present perspectives 
of our work in Section~\ref{sec:conclusion}.
Proofs of our results are postponed to the Appendix.

\bigskip

\noindent { \it Notation:} First, we state general notation.
Let $(X ,Y)$  be the generic data-structure taking its values in $\mathcal{X} \times \{0,1\}$
with distribution $\mathbb{P}$. Let $(X_{\bullet}, 
Y_{\bullet})$ be a random variable independent of $(X,Y)$ and with the same law as $(X,Y)$. 
The goal in classification is to predict the label $Y_{\bullet}$ given an observation of $X_{\bullet}$.
This is performed based on a classifier (or classification rule) $s$ which is a function mapping $\mathcal{X}$ onto 
$\{0, 1 \}$. Let $\mathcal{S}$ be the set of all classifiers.
The misclassification risk $R$ associated with $s \in \mathcal{S}$ is defined as
\begin{equation*}
R(s) = \mathbb{P}(s(X) \neq Y).
\end{equation*} 
Moreover, the minimizer of $R$ over $\mathcal{S}$ is the Bayes classifier, denoted by $s^*$, and is characterized by
\begin{equation*}
s^{*}(\cdot) = \1\{\eta^*(\cdot) \geq 1/2\},
\end{equation*}
where $\eta^*(x) = \mathbb{P}(Y = 1 | X = x)$ for $x \in \mathcal{X}$.
One of the most important quantities in our methodology is the function $f^*$ defined by
$f^*(\cdot) = \max\{\eta^*(\cdot), 1 - \eta^*(\cdot)\}$.
It will play the role of a score function.
\\
Let us now consider more specific notation related to the classification with reject option setting. Let $s\in \mathcal{S}$ be a classifier.
A confidence set $\Gamma_s$ associated with the classifier $s$
is defined as a measurable function that maps $\mathcal{X}$
onto $\{\{0\}, \{1\}, \{0,1\}\}$, such that for an example $\Xp$, the set $\Gamma_s(\Xp)$ can be either $\{s(\Xp)\}$ or $\{0,1\}$. 
We decide to classify the example $\Xp$, according to the label $s(\Xp)$, if 
$\card\left(\Gamma_s(\Xp)\right) = 1$. In the case where $\Gamma_s(\Xp) = \{0,1\}$, 
we decide to not classify (reject) the feature $\Xp$.
Let $\Gamma_s$ be a confidence set. The probability of classifying 
a feature is denoted by
\begin{equation}
	\label{Sec2:ProbaClass}
		\mathcal{P}\left(\Gamma_s\right) := \Pro\left(\card\left(\Gamma_s(\Xp)\right) = 1\right).
\end{equation}
In our approach, the probability of classifying $\mathcal{P}\left(\Gamma_s\right)$ is not viewed as a success or an error but simply as a parameter that we have to control.
Hence, the definition of a confidence set makes natural the following definition of the risk associated with $\Gamma_s$:
\begin{eqnarray}
	\label{Sec2:RiskConfidenceSet}
		\mathcal{R}\left(\Gamma_s\right) & = & 
		\mathbb{P}\left(\Yp \notin \Gamma_s(\Xp) | \card\left(\Gamma_s(\Xp)\right) = 1\right) \nonumber \\ 
		& = & \mathbb{P}\left(s(\Xp) \neq \Yp | \card\left(\Gamma_s(\Xp)\right) = 1\right).
\end{eqnarray}
The risk $\mathcal{R}\left(\Gamma_s\right)$ is the misclassification error risk of $s$
conditional to the event that $\Xp$ is classified.
Moreover,  for some $\varepsilon \in ]0,1]$, we say that, for two confidence sets $\Gamma_s$ and $\Gamma_{s^{'}}$ such that 
$\mathcal{P}\left(\Gamma_s\right)=\mathcal{P}\left(\Gamma_{s^{'}}\right) = \varepsilon$, the confidence set $\Gamma_s$ is ``better'' than $\Gamma_{s^{'}}$
if $\mathcal{R}\left(\Gamma_s\right) \leq \mathcal{R}\left(\Gamma_{s^{'}}\right)$.

\section{$\varepsilon$-confidence sets}
\label{sec:ourConfSet}

In the present section, we define a class of confidence sets referred as {\it $\varepsilon$-confidence sets} which are optimal according to the definition of risk~\eqref{Sec2:RiskConfidenceSet}.
We always keep in mind that the classification probability~\eqref{Sec2:ProbaClass} will be under control.
In Section~\ref{subsec:espConfSetDef}, we define and state the important properties of the class of $\eps$-confidence sets.
We then apply the $\eps$-confidence sets to the Gaussian mixture case in Section~\ref{subsec:gaussianMixture}.
We end up this section with a comparison to classifiers with reject option in Section~\ref{subsec:classR}.

\subsection{Definition and properties}
\label{subsec:espConfSetDef}

The definition of $\eps$-confidence sets relies on the Bayes classifier $s^{*}$ and the cumulative distribution function of $f^*(X)$.
\begin{definition}\label{def:Epsilonconfset}
	Let  $\varepsilon \in ]0,1]$, the $\varepsilon$-confidence set is defined as follows
	\begin{equation*}
		\Gamma_{\varepsilon}^{\bullet}(\Xp) =
		\begin{cases}
			\{s^{*}(\Xp)\} \;\;  {\rm if} \; F_f^*({f^*}(\Xp)) \geq 1-\varepsilon \\
			\{0,1\} \; \;\; \quad {\rm otherwise},
		\end{cases} 
	\end{equation*}
where $F_f^*$ is the cumulative distribution function of $f^*(X)$ and $f^*(\cdot)= \max \{ \eta^*(\cdot),1-\eta^*(\cdot)\}$.
\end{definition}
According to this definition, the construction of the $\varepsilon$-confidence sets relies on two important features.
First, if a label is assigned to a new feature $\Xp$ by the $\varepsilon$-confidence set, it is the 
one provided by the Bayes classifier $s^{*}$.
Second, we assign a label to a new data $\Xp$ if the corresponding score $f^*(\Xp)$ is large enough regarding the distribution of $f^*(X)$.
This is one of the key ideas behind conformal predictors introduced in~\cite{Vovk_livre}.

The following assumption is fundamental to establish theoretical guarantees.
{\it 
\begin{itemize}
\item[{\bf (A1)}] The cumulative distribution function $F^*_f$ of $f^*(X)$ is continuous.
\end{itemize}
}
One of the main motivations of the introduction of the $\varepsilon$-confidence set is that, if Assumption~(A1) 
holds, the procedure ensures an exact control of the probability~\eqref{Sec2:ProbaClass} of assigning a label
\begin{equation}
\label{eq:controlNonClass}
		\mathcal{P}(\Gamma^{\bullet}_{\varepsilon}) = \mathbb{P}(F_f^* \left( f^*(X_\bullet) \right)  \geq 1-\varepsilon) = \varepsilon.
\end{equation}
This happens since $F_f^*(f^*(\Xp))$ is uniformly distributed under Assumption (A1). 
Moreover, under this assumption as well, one can rewrite the $\varepsilon$-confidence sets in a different way.
Indeed,
for $\epsilon \in ]0,1[$, we have
\begin{equation*}
	F_f^*\left(f^*(\Xp)\right) \geq 1 - \varepsilon \Longleftrightarrow f^*\left(\Xp\right) \geq (F_f^*)^{-1}(1-\varepsilon),
\end{equation*}
where $(F_f^*)^{-1}$ denotes the generalized inverse of the cumulative distribution function $F_f^*$ (see \cite{VanderVaartLivre}). Therefore, if we set $\alpha_{\varepsilon} = (F_f^*)^{-1}(1-\varepsilon)$ 
for $\varepsilon \in ]0,1[$ and $\alpha_1 = 1/2$, Definition~\ref{def:Epsilonconfset} is equivalent to
\begin{equation}
	\label{eq:confoptEquivalent}
	\Gamma^{\bullet}_{\varepsilon}(\Xp) =
		\begin{cases}
			\{s^{*}(\Xp)\} \quad {\rm if} \; f^*(\Xp) \geq \alpha_{\varepsilon} \\
			\{0,1\} \;\;   \qquad {\rm otherwise}.
		\end{cases} 
\end{equation}
Next, we provide the most important property of the $\varepsilon$-confidence sets: 
\begin{proposition}
	\label{prop:prop2}
	Denote by $  \rm{\bf{\Gamma_{\varepsilon}}} $ the set $ \rm{\bf{\Gamma_{\varepsilon}}} = \{\Gamma_s; \; \mathcal{P}(\Gamma_s) = \varepsilon\}$.	
	Let Assumption (A1) be satisfied.
		\begin{enumerate}
			\item For any $\varepsilon \in ]0,1]$, the $\varepsilon$-confidence set satisfies the following property:
\begin{equation*}
\mathcal{R}\left(\Gamma_{\varepsilon}^{\bullet}\right) =  \min_{\Gamma_{s} \in {\rm{\bf{\Gamma_{\varepsilon}}}}} \mathcal{R}\left(\Gamma_s\right).
\end{equation*}	
			\item For $\varepsilon \in ]0,1]$ and for any $\Gamma_s \in \rm{\bf{\Gamma_{\varepsilon}}} $, the following holds
\begin{multline}
\label{decompRisk}
0\leq \mathcal{R}\left(\Gamma_s\right) - \mathcal{R}\left(\Gamma_{\varepsilon}^{\bullet}\right)  
= 
\dfrac{1}{\varepsilon} 
\left\{  \mathbb{E}\left[|2 \eta^*(\Xp) - 1 | 
\mathbf{1}_{\mathcal{C} }\right] 
+   
 \right.
\\ 
\mathbb{E}\left[|\eta^*(\Xp) - \alpha_{\varepsilon} | \mathbf{1}_{ \mathcal{A}_0 \cup \mathcal{B}_0 }   \right]  
+ \mathbb{E}\left[|1 - \eta^*(\Xp) - \alpha_{\varepsilon} | \mathbf{1}_{\mathcal{A}_1 \cup \mathcal{B}_1}   \right]\},
\end{multline}
where $\alpha_{\varepsilon} = (F_f^*)^{-1}(1-\varepsilon)$ and
{\footnotesize
\begin{eqnarray*}
\mathcal{A}_y & = &  \{f^*(\Xp) \geq \alpha_{\varepsilon},  \card(\Gamma_s(\Xp)) \neq 1, s^{*}(\Xp) \neq y\},  \;\; y = 0,1, \\
\mathcal{B}_y & = &  \{f^*(\Xp) < \alpha_{\varepsilon},  \card(\Gamma_s(\Xp)) = 1,  s(\Xp) \neq y\}, \;\; y = 0,1,\\
\mathcal{C} & = & \{f^*(\Xp) \geq \alpha_{\varepsilon},  \card\left(\Gamma_s(\Xp)\right) = 1,  s^{*}(\Xp) \neq s(\Xp)\}. \\
\end{eqnarray*}
}
\end{enumerate}
\end{proposition}
Several remarks can be made from Proposition~\ref{prop:prop2}.
First, for $\eps \in ]0,1]$, the $\eps$-confidence set is optimal in the sense that its risk is minimal over ${\rm{\bf{\Gamma_{\varepsilon}}}}$, the class
of all confidence sets that assign a label with probability $\eps$.  
Second, the excess risk of a confidence set is directly linked to the behavior of the function 
$f^*$ around $\alpha_\varepsilon$. This observation will play a major role in our main result related to rates of 
convergence in the next section.
Third, note that if we apply~\eqref{decompRisk} with $\varepsilon = 1$, which implies $\alpha_{\varepsilon} = 1/2$, we obtain the classical result
in classification
\begin{equation*}
R(s) - R(s^{*}) =  \mathbb{E}\left[|2 \eta^*(\Xp) - 1 | \mathbf{1}_{\{s^{*}(\Xp) \neq s(\Xp)\}}\right].
\end{equation*}
Let us conclude this section by stating a result that specifies the behavior of the risk
associated with the $\varepsilon$-confidence set w.r.t. the parameter $\varepsilon$:
\begin{proposition}
\label{prop:prop3}
The function $\varepsilon \mapsto \mathcal{R}(\Gamma_\varepsilon^{\bullet})$
is non decreasing on $]0,1]$.
\end{proposition} 
This result shows an expected fact: the larger the rejecting probability, the smaller the error. In particular
\begin{equation*}
\mathcal{R}(\Gamma_{\eps}^{\bullet}) \leq R(s^{*}) \;\;\, \forall \eps \in ]0,1].
\end{equation*}

\subsection{$\varepsilon$-confidence sets for Gaussian mixture}
\label{subsec:gaussianMixture}

In this section, we apply the $\varepsilon$-confidence set introduced in Definition~\ref{def:Epsilonconfset} 
to the particular case of Gaussian mixture model.
We set $\mathcal{X} = \mathbb{R}^d$ with $d\in \mathbb{N}\setminus \{0\}$.
Let us assume that the conditional distribution of $X$ given $Y$ is Gaussian and that, for simplicity,
the marginal distribution of $Y$ is Bernoulli with parameter $1/2$. To fix notation, we set
\begin{eqnarray*}
X | Y=0 \sim \mathcal{N}(\mu_0,\Sigma)
\quad \text{and} \quad
X | Y=1 \sim \mathcal{N}(\mu_1,\Sigma),
\end{eqnarray*}
where $\mu_0$ and  $\mu_1$ are vectors in $\mathbb{R}^d$ and $\Sigma$ is the commun covariance matrix. We assume that 
$\Sigma$ is invertible and denote by $\Vert \cdot \Vert_{ \Sigma^{-1}} $ the norm under $\Sigma^{-1}$: 
for any $\mu \in \mathbb{R}^d$ we have $\Vert \mu \Vert_{ \Sigma^{-1}}^2 = \mu^\top \Sigma^{-1} \mu$ where $\mu^\top$ stands for the transpose of $\mu$.
The following theorem establishes the classification error of the $\varepsilon$-confidence set 
$\Gamma_{\varepsilon}^{\bullet}$ in this framework.
\begin{proposition}
	\label{Th:RiskBayes}
		For all $\varepsilon \in ]0,1]$, we have
		$$
			\mathcal{R}(\Gamma_{\varepsilon}^{\bullet}) = \frac{    \mathbb{P}\left(    \Phi \left(Z  \right)     
			+    \Phi \left( Z  +  \Vert \mu_1  -    \mu_0 
			\Vert_{ \Sigma^{-1}}   \right) \leq  \varepsilon  \right) }{\varepsilon},
		$$
		where $Z$ is a standard normal random variable and $\Phi$ is the standard normal cumulative distribution function.
\end{proposition}
The proof of this proposition is postponed to the Appendix.
Interestingly, in the Gaussian mixture case, we get a close formula for the risk of the $\eps$-confidence set.
Moreover, this risk depends on $\Vert \mu_1  -    \mu_0  \Vert_{ \Sigma^{-1}} $ as in the binary classification framework which corresponds to the particular case $\varepsilon = 1$, where we do not use the reject option and where we get
\begin{equation*}
\mathcal{R}(\Gamma_{1}^{\bullet}) =  R(s^*) = 1 - \Phi\left(\frac{ \Vert \mu_1  - \mu_0 \Vert_{ \Sigma^{-1}  }  }{ 2 } \right).
\end{equation*}

\subsection{Relation with classifiers with reject option}
\label{subsec:classR}

The problem of classification with reject option has already been introduced in~\cite{Ch70}.
More recently the terminology of {\it classifiers with reject option} has been defined in~\cite{HW06}:
a classifier with reject option is a measurable function which 
maps $\mathcal{X}$ onto $\{0,1,R_e\}$ where the output $R_e$ means reject.
For a parameter $\alpha \in [1/2,1]$ and for $s_R$ a classifier with reject option, the risk function considered in~\cite{HW06} is
\begin{equation}
\label{eq:RiskDebuWeg}
L_{\alpha}(s_R) = \mathbb{P}\left(s_R(X) \neq Y \, , \,  s_R(X)  \neq R_e\right)    +   (1-\alpha) \, \mathbb{P}\left(s_R(X) = R_e\right) . 
\end{equation} 
This risk has been studied in the context of classification with reject option in the papers \cite{HW06,WY11} 
and references therein.
We notice that the above risk looks at rejecting as a part of the error in the same way as wrong classification.
The parameter $1 -\alpha$ controls the trade-off between these two "errors".
In other words, the parameter $1-\alpha$ is the cost of using the reject option.
This is a major difference with our point of view. Indeed, we recall that the probability
of rejection is a parameter in our setting, 
and then we do not include it in the risk~\eqref{Sec2:RiskConfidenceSet}. 
As a consequence, if we bound the risk~\eqref{Sec2:RiskConfidenceSet} while keeping under control
the probability of classifying~\eqref{Sec2:ProbaClass}, we are able to bound the risk~\eqref{eq:RiskDebuWeg}.
The reverse is not true. 
That is, controlling~\eqref{eq:RiskDebuWeg} does not provide any control on the probability of rejection
and one cannot avoid irrelevant use of the reject option is some situations.
This difference can be significant in some some practical situations where the knowledge of $\mathcal{P}\left(\Gamma_s\right)$ is a relevant information.
Indeed, when dealing with several examples to label, controlling this probability ensures the amount of the data we wish label.
Hence our methodology prevents from irrelevant use of the reject option.
For the same reason, a second important feature that differs between both  methodologies is that the comparison between two confidence sets is easier than the comparison
between two classifiers with reject option.
Indeed, for some $\varepsilon \in ]0,1]$, we say that, for two confidence sets $\Gamma_s$ and $\Gamma_{s^{'}}$ such that 
$\mathcal{P}\left(\Gamma_s\right)=\mathcal{P}\left(\Gamma_{s^{'}}\right) = \varepsilon$, the confidence set $\Gamma_s$ is ``better'' than $\Gamma_{s^{'}}$
if $\mathcal{R}\left(\Gamma_s\right) \leq \mathcal{R}\left(\Gamma_{s^{'}}\right)$.
As the study of the risk function $L_{\alpha}$  viewed in 
Equation~\eqref{eq:RiskDebuWeg} does not provide any control on the probability of classifying,
it is much more difficult to compare the performance of two classifiers with reject option on the set of labeled data.
This point will be made clear in Section~\ref{sec:simu} with the numerical experiment.

Let us consider the optimality by now: the paper \cite{HW06} also provides the optimal rule for the risk~\eqref{eq:RiskDebuWeg}.
For each $\alpha \in [1/2,1]$, the {\it Bayes rule with reject option} $s^*_{R_\alpha}$ is defined such that
\begin{equation*} 
L_{\alpha}(s^*_{R_\alpha}) = \min_{s_R}  L_{\alpha}(s_{R}), 
\end{equation*}
where the minimum is taken over all classifiers with reject option. It is 
characterized by 
\begin{equation}
\label{eq:WegClassBayes}
s^{*}_{R_{\alpha}}(\Xp) = 		\begin{cases}
			s^*(\Xp) \quad {\rm if} \; f^*(\Xp) \geq \alpha, \\
			R_e \qquad \;\;\ {\rm otherwise},
		\end{cases} 
\end{equation}
where $f^*(\cdot)= \max \{ \eta^*(\cdot),1-\eta^*(\cdot)\}$ as in our setting.
Obviously, the {\it Bayes rule with reject option} $s^*_{R_\alpha}$ can be written in term 
of confidence sets. This leads to an $\eps$-confidence set defined in the same way as in Equation~\eqref{eq:confoptEquivalent}.
However, there is an important difference: the main contribution of the present paper is to
provide a methodology to pick the parameter $\alpha_\eps$ in~\eqref{eq:confoptEquivalent}
such that the probability of classifying an example~\eqref{Sec2:ProbaClass} is exactly $\eps$.
The key to be able to do so is the use of the cumulative distribution function of $f^*(X)$.
In Section~\ref{subsec:plugInDef}, we will see that the data-driven counterpart of the $\eps$-confidence
defined in Definition~\ref{def:Epsilonconfset} also controls the probability of classifying.
Notably, this is possible in a semi-supervised way, that is, only using a set of
unlabeled data.

\section{Plug-in $\varepsilon$-confidence sets}
\label{sec:Estimateion}
%
This section is devoted to the study the data driven counterpart of the $\eps$-confidence sets provided by plug-in rule.
We provide the construction 
of the plug-in methods in
Section~\ref{subsec:plugInDef}. Their asymptotic consistency as well as rates of convergence are given in Section~\ref{subsec:plugInAs}. 
%
\subsection{Definition of the plug-in $\eps$-confidence sets}
\label{subsec:plugInDef}
%
For $\eps \in ]0,1]$, the construction of our plug-in $\eps$-confidence set relies on a previous 
estimator of the regression function $\eta^*$.
To this end, we introduce a first dataset, $\mathcal{D}_n$, which consists of $n$ independent copies of $(X,Y)$. 
The dataset $\mathcal{D}_n $ is used to estimate the function $\eta$ and therefore the functions $f^*$ and $s^{*}$ as 
well. 
Let us denote by $\hat{\eta}$, $\hat{f}(\cdot) = \max(\hat{\eta}(\cdot), 1-\hat{\eta}(\cdot))$ and $\hat{s} = 
\1_{\{\hat{\eta}(\cdot) \geq 1/2\}}$ the 
estimators of $\eta^*$, $f^*$ and $s^{*}$ respectively.
Thanks to these estimations,  a data driven approximation of the $\varepsilon$-confidence set given in 
Definition~\ref{def:Epsilonconfset} can be
\begin{equation}
\label{def:pseudoestim}
	\widetilde{\Gamma}_{\varepsilon}^{\bullet}(\Xp) =
	\begin{cases}
		\{\hat{s}(\Xp)\} \;\; {\rm if} \; F_{\hat{f}}\left({\hat{f}}(\Xp)\right) \geq 1-\varepsilon \\
		\{0,1\} \; \; \quad {\rm otherwise},
	\end{cases} 
\end{equation}
where $F_{\hat{f}}$ is the cumulative distribution function of $\hat{f}(X)$.
Hence, $\widetilde\Gamma_{\varepsilon}^{\bullet}(\Xp)$ invokes the cumulative distribution function 
$F_{\hat{f}}$,  which is unknown and therefore needs to be estimated.
We then consider a second dataset, independent of $\mathcal{D}_n$, denoted by 
$\mathcal{D}_N = \{ X_i, i = 1, \ldots,N \}$ where $X_1, 
\ldots, X_N$ are independent copies of $X$.
Based on $\mathcal{D}_N$, we estimate the cumulative function $F_{\hat{f}}$ by the empirical cumulative distribution function of $\hat{f}(X)$
denoted by $\hat{F}_{\hat{f}}$. 
Now, we can define the plug-in $\varepsilon$-confidence set:
\begin{definition}\label{def:Epsilonconfset2}
	Let  $\varepsilon \in ]0,1]$ and $\hat{\eta}$ be any estimator of $\eta^*$, the plug-in $\varepsilon$-confidence set is defined as follows:
	\begin{equation*}
		\widehat{\Gamma}_{\varepsilon}^{\bullet}(\Xp) =
		\begin{cases}
			\{ \hat{s}(\Xp) \} \;\; {\rm if} \; \hat{F}_{\hat{f}}(\hat{f}(\Xp)) \geq 1-\varepsilon \\
			\{0,1\} \;\;  \quad {\rm otherwise},
		\end{cases} 
	\end{equation*}
where $\hat{f}(\cdot)= \max \{ \hat{\eta}(\cdot),1-\hat{\eta}(\cdot)\}$ and $\hat{F}_{\hat{f}}(\hat{f}
(\Xp))= \frac{1}{N} \sum_{i=1}^N \mathbf{1}_{\left\{   \hat{f}(X_i) \leq 
\hat{f}(\Xp)\right\}}$.
\end{definition}
\begin{remark} 
The samples $\mathcal{D}_n$ and $\mathcal{D}_N$ play completely different roles. The sample $\mathcal{D}_n$ is used to estimate $\eta^*$ and then must consist of labeled observations. The second dataset $\mathcal{D}_N$, used to estimate the cumulative function $F_{\hat{f}}$, requires only a set 
of {unlabeled} observations. Hence, the construction of the plug-in $\eps$-confidence sets does not require more labeled examples than in the classical classification setting.   
This is particularly interesting in some practical situations where the number of labeled examples is small while a large number 
of unlabeled observations is available.
\end{remark}

%
\subsection{Theoretical performance}
\label{subsec:plugInAs}
%

This section is devoted to  assessing the asymptotic performances of the {plug-in $\varepsilon$-confidence set}. The symbols $\mathbf{P}$ and $\mathbf{E}$
stand for generic probability and expectation, respectively.
Let $\eps \in ]0,1]$, and $\widehat{\Gamma}_{\varepsilon}^{\bullet}$ be a plug-in $\eps$-confidence set.
 We define the risk of $\widehat{\Gamma}_{\varepsilon}^{\bullet}$ by the natural quantity 
\begin{equation}
\label{riskPluggEstim}
	\mathbf{R} \left(\widehat{\Gamma}_{\varepsilon}^{\bullet}\right)= \mathbf{P}\left(\hat{s}(\Xp) \neq \Yp | \hat{F}_{\hat{f}}(\hat{f}(\Xp))  \geq 1- 
	\varepsilon \right).
\end{equation}
Note that we use here the notation $\mathbf{R}$ rather than the previous one $\mathcal{R}$ 
given by~\eqref{Sec2:RiskConfidenceSet} to stress that the probability $\mathbf{P}$ is taken under the law of 
$(\mathcal{D}_n,\mathcal{D}_N,(\Xp,\Yp))$ instead of just $(\Xp,\Yp)$.
Throughout this section we assume the following condition on the cumulative distribution function $F_{\hat{f}}$, which is analogous to Assumption (A1).
 However, this assumption relies on the estimator $\hat{f}$ and then is not restrictive since it can be chosen by the statistician.
{\it 
\begin{itemize}
\item[{\bf (A2)}] The cumulative distribution function $F_{\hat{f}}$ of $\hat{f}(X)$ is continuous.
\end{itemize}
}
We also define the risk of the oracle counterpart $\widetilde{\Gamma}_{\varepsilon}^{\bullet}$ of $\widehat{\Gamma}_{\varepsilon}^{\bullet}$
\begin{equation*}
\mathbf{R} \left(\widetilde{\Gamma}_{\varepsilon}^{\bullet}\right)= \mathbf{P}\left(\hat{s}(\Xp) \neq \Yp |{F}_{\hat{f}}(\hat{f}(\Xp))  \geq 1- \varepsilon \right).
\end{equation*}
The objective of this section is to prove both that
\begin{eqnarray}
\label{eq:convProba}
\mathbf{P} \left(\hat{F}_{\hat{f}}(\hat{f}(\Xp))  \geq 1- 
	\varepsilon \right)
   \rightarrow  & \eps , &   {\rm and} \\
  \mathbf{R} \left(\widehat{\Gamma}_{\varepsilon}^{\bullet}\right)
   \rightarrow  & \mathcal{R}\left(\Gamma_{\varepsilon}^{\bullet}\right), &  n,N \to  + \infty, \nonumber
\end{eqnarray}
and to derive rates for these convergences. Since $\mathcal{D}_n$ is dedicated to the estimation of $\eta^*$ and $\mathcal{D}_N$ to the estimation of 
$F_{\hat{f}}$, we prove that
\begin{eqnarray}
\mathbf{R} \left(\widehat{\Gamma}_{\varepsilon}^{\bullet}\right) - \mathbf{R} \left(\widetilde{\Gamma}_{\varepsilon}^{\bullet}\right) & \rightarrow & 0 , \;\; {\rm and} 
\label{conv2}\\
\mathbf{R} \left(\widetilde{\Gamma}_{\varepsilon}^{\bullet}\right) - \mathcal{R}\left(\Gamma_{\varepsilon}^{\bullet}\right) & \rightarrow &  0, \label{conv1} 
\end{eqnarray}
when both $n$ and $N$ go to infinity.
The convergences~\eqref{eq:convProba} and \eqref{conv2} relies on the Dvoretzky--Kiefer--Wolfowitz inequality~\cite{Mass90} while the convergence~\eqref{conv1} relies on the following inequality.
%
%
\begin{proposition}
\label{propo:presqconfset}
For all $\varepsilon \in ]0,1]$, the following inequality holds under Assumptions (A1) and (A2) 
\begin{multline*}
0 \leq \mathbf{R} \left(\widetilde{\Gamma}_{\varepsilon}^{\bullet}\right) - 
\mathcal{R}\left(\Gamma_{\varepsilon}^{\bullet}\right) \leq 
\dfrac{1}{\varepsilon} \{\mathbf{E}\left[|\eta^*(\Xp) - \alpha_{\varepsilon}| \mathbf{1}_{\{|\hat{\eta}(\Xp) - 
\eta^*(\Xp)| \geq |\eta^*(\Xp) - \alpha_{\varepsilon}|\}}\right] \\
+ \mathbf{E}\left[|1 - \eta^*(\Xp) - \alpha_\varepsilon| \mathbf{1}_{\{|\hat{\eta}(\Xp) - \eta^*(\Xp)| \geq |1 - \eta^*(\Xp) 
- \alpha_{\varepsilon}|\}} \right] \\
+ \alpha_{\varepsilon} |F_{\hat{f}}(\alpha_{\varepsilon}) - F_{f}^*(\alpha_{\varepsilon})|\},
\end{multline*}
where $\alpha_{\varepsilon} = (F_{{f}}^*)^{-1}(1-\varepsilon)$.
\end{proposition} 
%
%
The proof of Proposition~\ref{propo:presqconfset} relies on Proposition~\ref{prop:prop2}.
For $\eps \in ]0,1]$, Proposition~\ref{propo:presqconfset} evaluates the loss of performance using the confidence set $\widetilde{\Gamma}_{\varepsilon}^{\bullet}$ instead 
of the $\varepsilon$-confidence set $\Gamma_{\varepsilon}^{\bullet}$. We can distinguish two parts in the upper bound. One part is linked to the classification with reject 
option setting provided by~\cite{HW06} and then depends on the behavior of the 
function $f$ around $\alpha_\varepsilon$. Note that the same quantity is obtained by~\cite{HW06}.
The second part $\alpha_{\varepsilon} |F_{\hat{f}}(\alpha_{\varepsilon}) - F_{f}^*(\alpha_{\varepsilon})|$ is related to our proposed confidence set and is due to the 
approximation of $F_{f}^*$ by $F_{\hat{f}}$.

Observe that when $\varepsilon = 1$ one can recover a classical inequality in the classification setting. Indeed, in this case, $\alpha_{\varepsilon} = 1/2$
and $F_{\hat{f}}(1/2) =  F_{f}^*(1/2)$. Hence, we obtain
\begin{equation*}
\mathbf{R} \left(\widetilde{\Gamma}_{1}^{\bullet}\right) - \mathcal{R}\left(\Gamma_{1}^{\bullet}\right) \leq
\mathbf{E}\left[|2\eta^*(\Xp) - 1|\mathbf{1}_{\{|\hat{\eta}(\Xp) - \eta^*(\Xp)| \geq |\eta^*(\Xp)-1/2|\}}\right]. 
\end{equation*}

Finally, we state our main result which describes the asymptotic behavior of our plug-in $\eps$-confidence sets:

\begin{theorem}
	\label{thm:mainThm}
\begin{enumerate}
\item If $\hat{\eta}(\Xp) \rightarrow \eta^*(\Xp)$ in probability when $n \rightarrow +\infty$, then for any $\varepsilon \in ]0,1]$
	\begin{equation*}
		\mathbf{P} \left(\hat{F}_{\hat{f}}(\hat{f}(\Xp))  \geq 1- 
		\varepsilon \right)
		\rightarrow   \eps,
	\end{equation*}
and		
	\begin{equation*}
		\mathbf{R} \left(\widehat{\Gamma}_{\varepsilon}^{\bullet}\right)
		 -  \mathcal{R}\left(\Gamma_{\varepsilon}^{\bullet}\right)  \rightarrow  0,	
	\end{equation*}
when both $n$ and $N$ go to infinity.
\item For any $\varepsilon \in ]0,1]$, assume that there exist $C_1 < \infty$ and $\gamma_\varepsilon > 0$ such that
	\begin{equation}
	\label{MargeAlpha}
	\mathbf{P} \left(  \vert f(X) - \alpha_\varepsilon   \vert \leq t  \right) \leq  C_1 t^{\gamma_\varepsilon}, \qquad \forall t>0.
	\end{equation}
Assume also that there exist a sequence of positive numbers $a_n \rightarrow +\infty$ and some positive constants $C_2, C_3$ such that
	\begin{equation}
		\label{devTsyb}
 	\mathbf{P}\left(|\hat{\eta}(x) - \eta(x)| \geq t \right) \leq C_2 \exp\left(- C_3 a_n t^{2} \right), \qquad 
 	\forall t > 0, \ \forall x\in \mathcal{X}.
	\end{equation}
Then we have 
	\begin{equation}
	\label{eq:eqRateProba}
	\mathbf{P} \left(\hat{F}_{\hat{f}}(\hat{f}(\Xp))  \geq 1- 
		\varepsilon \right)
		=    \eps + O(N^{-1/2}),
	\end{equation}
and
	\begin{equation}
	\label{eq:eqRate}
	\mathbf{R} \left(\widehat{\Gamma}_{\varepsilon}^{\bullet}\right)
		 -  \mathcal{R}\left(\Gamma_{\varepsilon}^{\bullet}\right) = O(a_n^{-\gamma_\varepsilon/2}) + O(N^{-1/2}).
	\end{equation}
\end{enumerate}
\end{theorem}

The proof of this result is postponed to the Appendix. Theorem~\ref{thm:mainThm} states that if the estimator of 
$\eta^*$ is consistent, then asymptotically the plug-in 
$\varepsilon$-confidence set is level-$\eps$-confidence set and performs as well as the $\varepsilon$-confidence set.
Moreover, several observations can be made according to the second point of the theorem.
First of all, we mention the rate of convergence~\eqref{eq:eqRateProba} does not require any of the assumptions~\eqref{MargeAlpha} and~\eqref{devTsyb}. It only needs the consistency of the estimator of $\eta^*$.
Second, we mention that the assumption~\eqref{MargeAlpha} has already been introduced in the classification with reject 
option setting in~\cite{HW06}. 
It is analogous to Tsybakov's margin condition in \cite{TsybMarre,AT07} introduced in the classification 
framework. 
We point out here the fact that if $\eta^*(\Xp)$ has a density w.r.t. the Lebesgue measure, the 
assumption~\eqref{MargeAlpha} is satisfied with $\gamma_\varepsilon \geq 1$ for any $\varepsilon \in ]0,1]$.
Third, if $\gamma_\varepsilon \gg 1$, we can get faster rate of convergence.
However, this rate cannot be better than $O(N^{-1/2})$ which is the term due to the estimation of the cumulative 
distribution function $F_{\hat{f}}$.
This term is however not limiting. Indeed, recalling that the sample size $N$ refers to the dataset $\mathcal{D}_N$ 
which can consist only of 
unlabeled observations, getting large $N$ is not a big issue.
Hence, we can consider the first term $ O(a_n^{-\gamma_\varepsilon/2}) $ as the leading term 
in~\eqref{eq:eqRate}.
This term relies on Proposition~\ref{propo:presqconfset} and on the assumption~\eqref{devTsyb} which 
is crucial to establish our rate of convergence.
Note that various estimators satisfy this condition such as kernel estimators (see~\cite{AT07}, for
more details).

\section{Numerical results}
\label{sec:simu}

In this section, we evaluate the plug-in $\varepsilon$-confidence sets numerically.
Moreover, we indicate the importance of Assumptions (A1) and (A2).

\subsection{Under Assumptions (A1)-(A2)}
\label{subsec:a1a2}
In this section both of the cumulative distribution functions $F_f$ and $F_{\hat{f}}$ are continuous.
We generate $(X,Y)$ according to the following models.
\begin{itemize}
\item Model~1:
\begin{enumerate}
\item  the feature $X \overset{\mathcal{L}}{=} (U_1, \ldots, U_{10})$, where $U_i$ are i.i.d from a 
uniform distribution on $[0,1]$;
\item conditional on $X$,  the label $Y$ is drawn according to a Bernoulli distribution with parameter 
$\eta^*(X)$ 
defined by
$\text{logit}(\eta^*(X)) = X^{1} - X^{2} - X^{3} + X^{9}$, where $X^{j}$ is the $j^{\text{th}}$
component of $X$.
\end{enumerate}
\item Model~2:
\begin{enumerate}
\item the feature $X \overset{\mathcal{L}}{=} (\mathcal{N}_1, \mathcal{N}_{2}, \mathcal{N}_{3})$, where $\mathcal{N}_i$ are i.i.d from
standard Gaussian distribution;
\item conditional on $X$, the label $Y$ is drawn according to a Bernoulli distribution with parameter
$\eta^*(X)$ defined by
$\text{logit}(\eta^*(X)) = (X^{1})^{2} + \frac{X^{2}}{2} + \sin(X^{1} + X^{3}) + 3X^{3} $. 
\end{enumerate}
\end{itemize}

The first model leads to a classification problem which is quite difficult.
Indeed, using a large dataset of features, we evaluate the distribution function of $\eta^*(X)$, and then  obtain 
that $\mathbb{P}(\eta^*(X) \in [0.4,0.6]) \simeq 0.5$.
On the contrary, estimating $\eta^*$ is easy since $\text{logit}(\eta^*(X))$ is a linear function of $X$.
Model~2 provides a more simple classification problem: the estimation of the distribution function of $\eta^*(X)$ leads 
to $\mathbb{P}(\eta^*(X) \in [0.4,0.6]) \simeq 0.15$. On the other side, the estimation of $\eta^*$ is a little more 
tricky. 

In order to illustrate our convergence result, we first provide estimation of the risk $\mathcal{R}$
for the $\varepsilon$-confidence sets.
More precisely, for each model and each $\varepsilon \in \{\frac{k}{10}, \;\; k \in \{1, \ldots, 10\}\}$,
we repeat $B = 100$ times the following steps:
\begin{enumerate}
\item[$i)$] simulate two data sets $\mathcal{D}_{N}$ and $\mathcal{D}_{K}$ according to the considered model with $N 
= 1000$ 
and $K = 1000$;
\item[$ii)$] based on $\mathcal{D}_{N}$, we compute the empirical cumulative distribution of $f^*(X)$ (this step 
requires only the features);
\item[$iii)$] finally, we compute, over $\mathcal{D}_{K}$, the empirical counterparts $\mathcal{R}_{K}$ of the risk 
$\mathcal{R}$  of the 
$\varepsilon$- confidence set using the empirical cumulative distribution of $f^*(X)$ instead of $F_{f}^*$. We also 
compute the proportion of classified instances $\mathcal{P}_K$.  
\end{enumerate}
From these experiments, we compute the mean and standard deviation of $\mathcal{R}_K$ and $\mathcal{P}_K$.
The results are reported in Table~\ref{table:tableb} and illustrated in Figure~\ref{fig:results}.
Next, for each model and each $\varepsilon \in \{\frac{k}{10}, \;\; k \in \{1, \ldots, 10\}\}$, we estimate the risk $\mathbf{R}$ for the plug-in 
$\varepsilon$-confidence set. We propose to use three popular classification methods for the 
estimation of $\eta^*$: {\it random forest, logistic regression} and {\it kernel rule} based on the Gaussian kernel and window parameter equal to $1$.
We perform the following simulation scheme.
We repeat independently $B$ times the following steps:
\begin{enumerate}
\item[$i)$] simulate three dataset $\mathcal{D}_{n}, \mathcal{D}_{N}, 
\mathcal{D}_{K}$ according to the considered model; 
\item[$ii)$] based on $\mathcal{D}_{n}$, we compute an estimate, denoted by 
$\hat{f}$, of $f^*$ with the random forest, the logistic regression or kernel rule procedure;
\item[$iii)$] based on $\mathcal{D}_{N}$, we compute the empirical cumulative distribution of 
$\hat{f}(X)$ (we recall that this step requires a dataset which contains only the features);
\item[$iv)$] finally, over $\mathcal{D}_{K}$, we compute the empirical counterpart $\mathbf{R}_K$ of the risk  
$\mathcal{R}$ and the proportion $\mathcal{P}_K$ of the data which are not rejected. 
\end{enumerate}

From these results, we compute the means and standard deviations of both empirical risks and proportions of classified instances
for $n \in \{100,1000\}$. We fix $N = 100$ and $K = 1000$. The results are illustrated in Figure~\ref{fig:results} and provided in 
Table~\ref{table:tablep} and~\ref{table:epstablep}.
\begin{table}
\begin{center}
\begin{tabular}{l || cc || cc}
\multicolumn{1}{c}{}& \multicolumn{2}{c}{Model 1} & \multicolumn{2}{c}{Model 2}\\
\hline\noalign{\smallskip}
$\varepsilon$ \; & $\mathcal{R}_{K}$ \; & $\mathcal{P}_K$ \; & $\mathcal{R}_{K}$ \; & $\mathcal{P}_K$ \\
\noalign{\smallskip}
\hline
\noalign{\smallskip}
1 \;& 0.39    (0.01)  \;& 1.00 (0.00) \;\; &\;\; 0.22 (0.01) \;& 1.00 (0.00) \\
0.9 \;& 0.38  (0.02) \;& 0.90 (0.01)  \;\; &\;\; 0.19 (0.01) \;& 0.90 (0.01)\\
0.8 \;&  0.37 (0.02) \;& 0.80 (0.02)  \;\; &\;\; 0.16 (0.01) \;& 0.80 (0.02)\\
0.7 \;&  0.35 (0.02) \;& 0.70 (0.02) \;\;  &\;\; 0.14 (0.01) \;& 0.70 (0.02)\\
0.6 \;& 0.34  (0.02) \;& 0.60 (0.02) \;\;  &\;\; 0.12 (0.01) \;& 0.60 (0.02) \\
0.5 \;& 0.33  (0.02) \;& 0.50 (0.02) \;\;  &\;\; 0.09 (0.01) \;& 0.50 (0.02)\\
0.4 \;& 0.31  (0.02) \;& 0.40 (0.02) \;\;  &\;\; 0.07 (0.01) \;& 0.40 (0.02) \\
0.3 \;& 0.29  (0.03) \;& 0.30 (0.02) \;\;  &\;\; 0.05 (0.01) \;& 0.30 (0.02) \\
0.2 \;& 0.27  (0.03) \;& 0.20 (0.02) \;\;  &\;\; 0.03 (0.01) \;& 0.20 (0.02) \\
0.1 \;& 0.24  (0.03) \;& 0.10 (0.01) \;\;  &\;\; 0.02 (0.01) \;& 0.10 (0.01)\\
\hline
\end{tabular}
\end{center}
\caption{\label{table:tableb} 
For each of the $B = 100$ repetitions and each model, we derive the estimates $\mathcal{R}_K$ of the risk
and the estimated proportions of classified instances $\mathcal{P}_K$
of the $\varepsilon$-confidence sets w.r.t. $\eps$.
We compute the means and standard deviations (between parentheses) over the $B = 100$ repetitions.
Left: the data are generated according to Model~1 -- Right: the data are generated according to Model~2.}
\end{table}

\begin{table}
\begin{center}
\footnotesize{
\begin{tabular}{l || ccc || ccc} 
\multicolumn{1}{c}{} & \multicolumn{6}{c}{{Model~1}}\\
\hline 
\multicolumn{1}{c}{} &  \multicolumn{3}{c}{$n = 100$} & \multicolumn{3}{c}{$n = 1000$} \\
\hline\noalign{\smallskip}
 $\varepsilon$ \;\;\;  & \;\;   \texttt{rforest} \;\;\; &  \texttt{logistic reg} \;\;\;& \texttt{kernel} & \;\;   \texttt{rforest} \;\;\; &  \texttt{logistic reg} \;\;\;& 
 \texttt{kernel} \\
\noalign{\smallskip}
\hline
\noalign{\smallskip}
 1    & 0.45 (0.02) & 0.43 (0.03) & 0.47 (0.03) \;\;&\;\; 0.42 (0.02) & 0.39 (0.02) & 0.42 (0.03) \\
 0.9  & 0.45 (0.02) & 0.42 (0.03) & 0.45 (0.03) \;\;&\;\;  0.41 (0.02) & 0.38 (0.02) & 0.41 (0.03) \\
 0.8  & 0.44 (0.02) & 0.42 (0.03) & 0.45 (0.03) \;\;&\;\;  0.40 (0.02) & 0.37 (0.02) & 0.39 (0.03) \\
 0.7  & 0.44 (0.03) & 0.41 (0.03) & 0.44 (0.03) \;\;&\;\;  0.39 (0.02) & 0.36 (0.02) & 0.38 (0.03) \\
 0.6  & 0.43 (0.03) & 0.40 (0.03) & 0.43 (0.03) \;\;&\;\;  0.38 (0.02) & 0.35 (0.02) & 0.37 (0.03) \\
 0.5  & 0.42 (0.03) & 0.39 (0.03) & 0.42 (0.03) \;\;&\;\;  0.37 (0.02) & 0.34 (0.02) & 0.36 (0.03) \\
 0.4  & 0.41 (0.03) & 0.38 (0.04) & 0.40 (0.04) \;\;&\;\; 0.36 (0.03) & 0.32 (0.03) & 0.34 (0.02) \\
 0.3  & 0.41 (0.04) & 0.37 (0.04) & 0.39 (0.04) \;\;&\;\; 0.35 (0.03) & 0.30 (0.03) & 0.33 (0.04) \\
 0.2  & 0.40 (0.04) & 0.35 (0.05) & 0.37 (0.05) \;\;&\;\;  0.34 (0.03) & 0.28 (0.04) &  0.30 (0.04) \\
 0.1  & 0.38 (0.06) & 0.33 (0.06) & 0.35 (0.06) \;\;&\;\;  0.32 (0.05) & 0.25 (0.05) & 0.27 (0.05) \\
\hline
\end{tabular}
}

\vspace*{0.25cm}

\footnotesize{
\begin{tabular}{l || ccc || ccc} 
\multicolumn{1}{c}{} & \multicolumn{6}{c}{{Model~2}}\\
\hline 
\multicolumn{1}{c}{} &  \multicolumn{3}{c}{$n = 100$} & \multicolumn{3}{c}{$n = 1000$} \\
\hline\noalign{\smallskip}
  $\varepsilon$ \;\;\;  & \;\;   \texttt{rforest} \;\;\; &  \texttt{logistic reg} \;\;\;& \texttt{kernel} & \;\;   \texttt{rforest} \;\;\; &  \texttt{logistic reg} \;\;\;& 
  \texttt{kernel} \\
\noalign{\smallskip}
\hline
\noalign{\smallskip}
 1    & 0.26 (0.02) & 0.24 (0.01) & 0.27 (0.05) \;\;&\;\;  0.24 (0.01) & 0.22 (0.01) & 0.23 (0.02) \\
 0.9  & 0.24 (0.02) & 0.21 (0.02) & 0.25 (0.05) \;\;&\;\;  0.22 (0.01) & 0.20 (0.01) & 0.20 (0.01) \\
 0.8  & 0.21 (0.02) & 0.18 (0.02) & 0.22 (0.04) \;\;&\;\;  0.19 (0.02) & 0.17 (0.02) & 0.18 (0.02) \\
 0.7  & 0.19 (0.02) & 0.16 (0.02) & 0.19 (0.04) \;\;&\;\;  0.16 (0.02) & 0.14 (0.02) & 0.15 (0.02) \\
 0.6  & 0.18 (0.02) & 0.13 (0.02) & 0.16 (0.04) \;\;&\;\;  0.15 (0.02) & 0.12 (0.02) & 0.13 (0.02) \\
 0.5  & 0.16 (0.03) & 0.11 (0.02) & 0.14 (0.04) \;\;&\;\;  0.12 (0.02) & 0.10 (0.02) & 0.11 (0.02) \\
 0.4  & 0.15 (0.03) & 0.09 (0.02) & 0.11 (0.03) \;\;&\;\;  0.11 (0.02) & 0.08 (0.02) & 0.08 (0.02) \\
 0.3  & 0.13 (0.03) & 0.07 (0.02) & 0.08 (0.03) \;\;&\;\;  0.09 (0.02) & 0.06 (0.02) & 0.06 (0.02) \\
 0.2  & 0.12 (0.03) & 0.05 (0.02) & 0.06 (0.02) \;\;&\;\;  0.08 (0.02) & 0.04 (0.01) &  0.04 (0.02) \\
 0.1  & 0.10 (0.04) & 0.03 (0.02) & 0.04 (0.02) \;\;&\;\;  0.06 (0.03) & 0.02 (0.01) & 0.02 (0.02) \\
\hline
\end{tabular}
}

\end{center}
\caption{\label{table:tablep}
For each of the $B = 100$ repetitions and each model, we derive the estimated 
risks $\mathbf{R}_K$ of three different plug-in $\varepsilon$-confidence sets w.r.t. $\eps$
and to the sample size $n$.
We compute the means and standard deviations (between parentheses) over the $B = 100$ repetitions. 
For each $\eps$ and each $n$, the plug-in $\varepsilon$-confidence sets are based on, from left to right, 
\texttt{rforest}, \texttt{logistic reg} and \texttt{kernel}, 
which are respectively the random forest, the logistic regression and the kernel rule procedures.
Top: the data are generated according to Model 1 -- Bottom: the data are generated according to Model 2.}
\end{table}
\begin{figure}
\begin{center}
\begin{tabular}{cc}
\includegraphics[width = 6cm, scale=1]{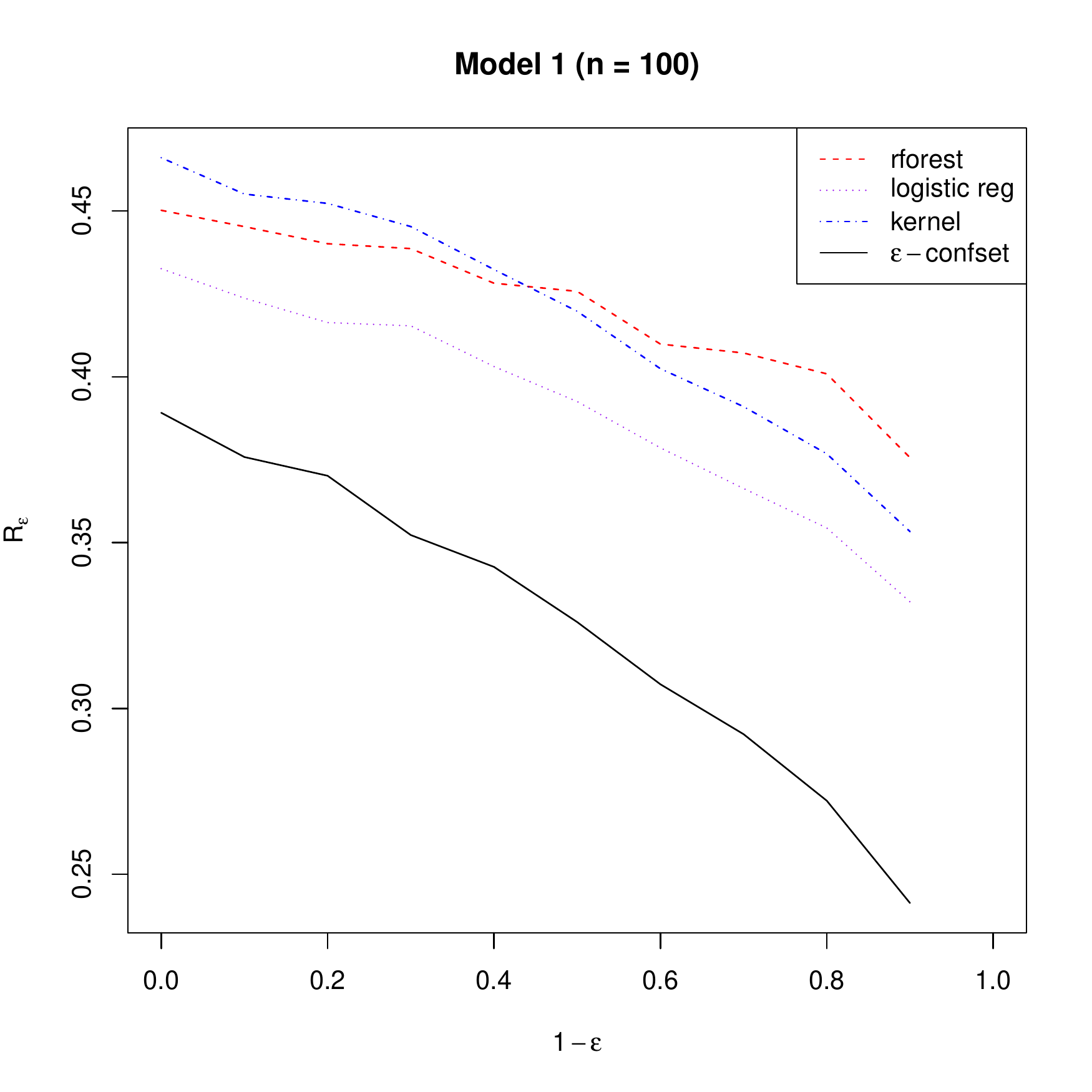} & \includegraphics[width = 6cm, scale=0.1]{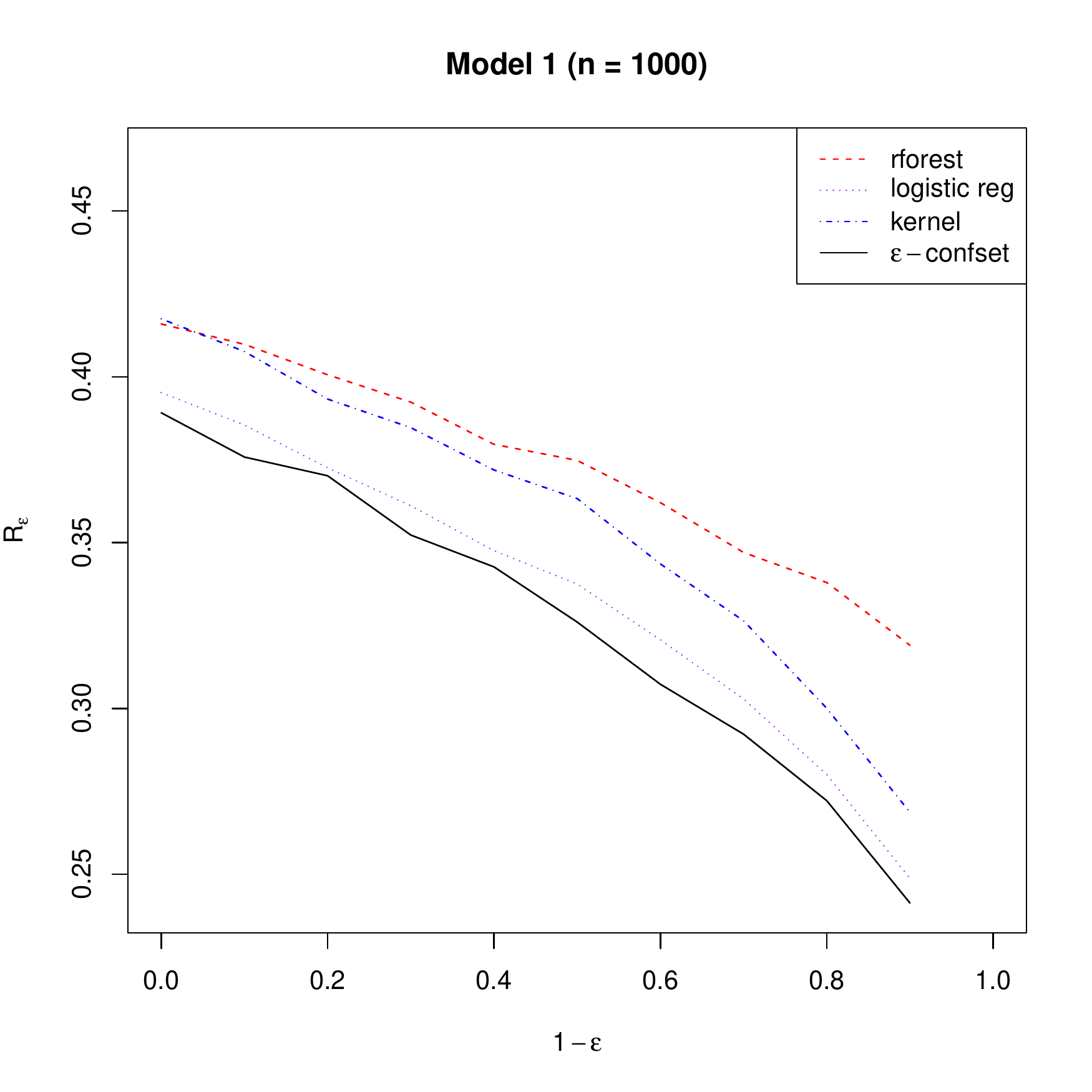}\\
\includegraphics[width = 6cm, scale=1]{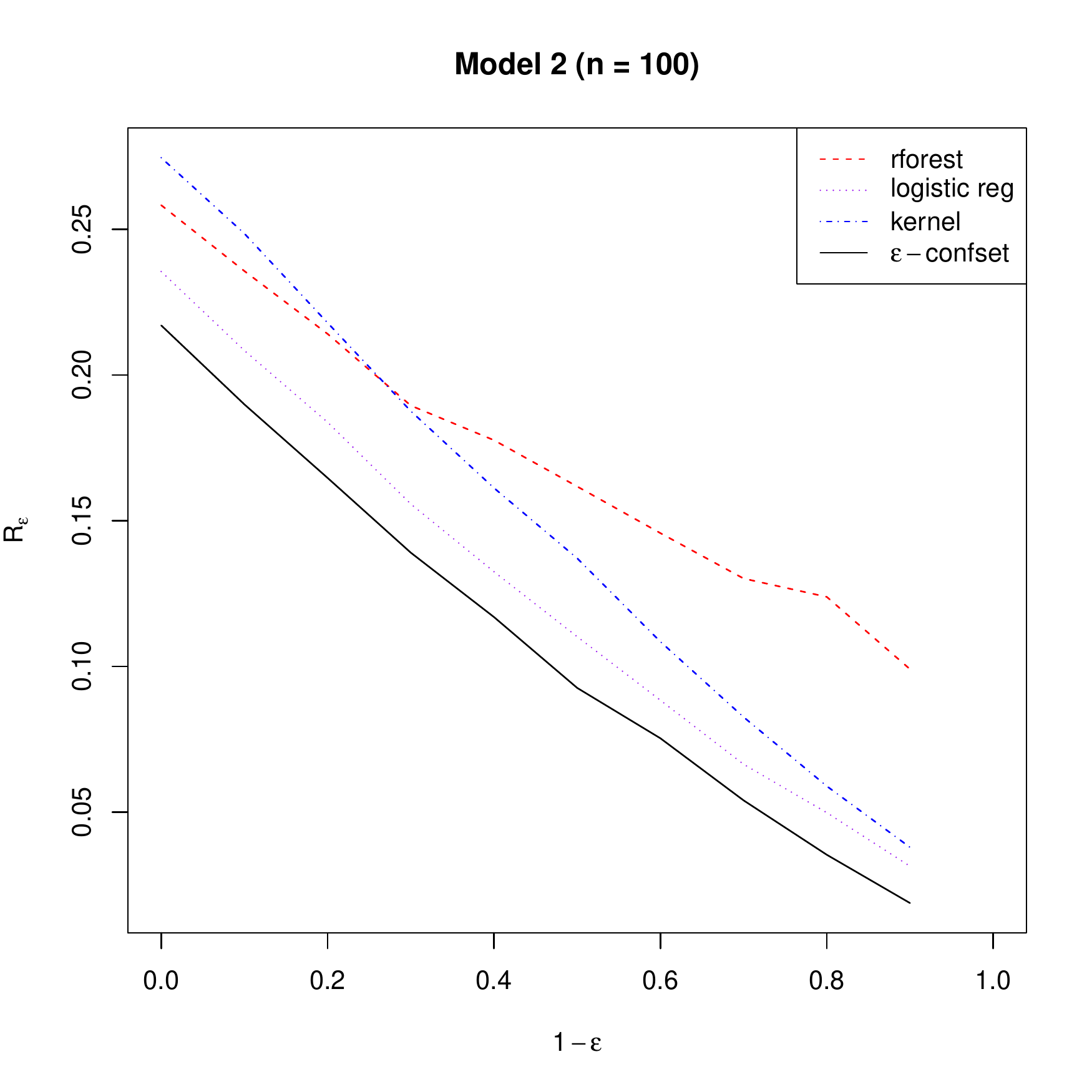} & \includegraphics[width = 6cm, scale=0.1]{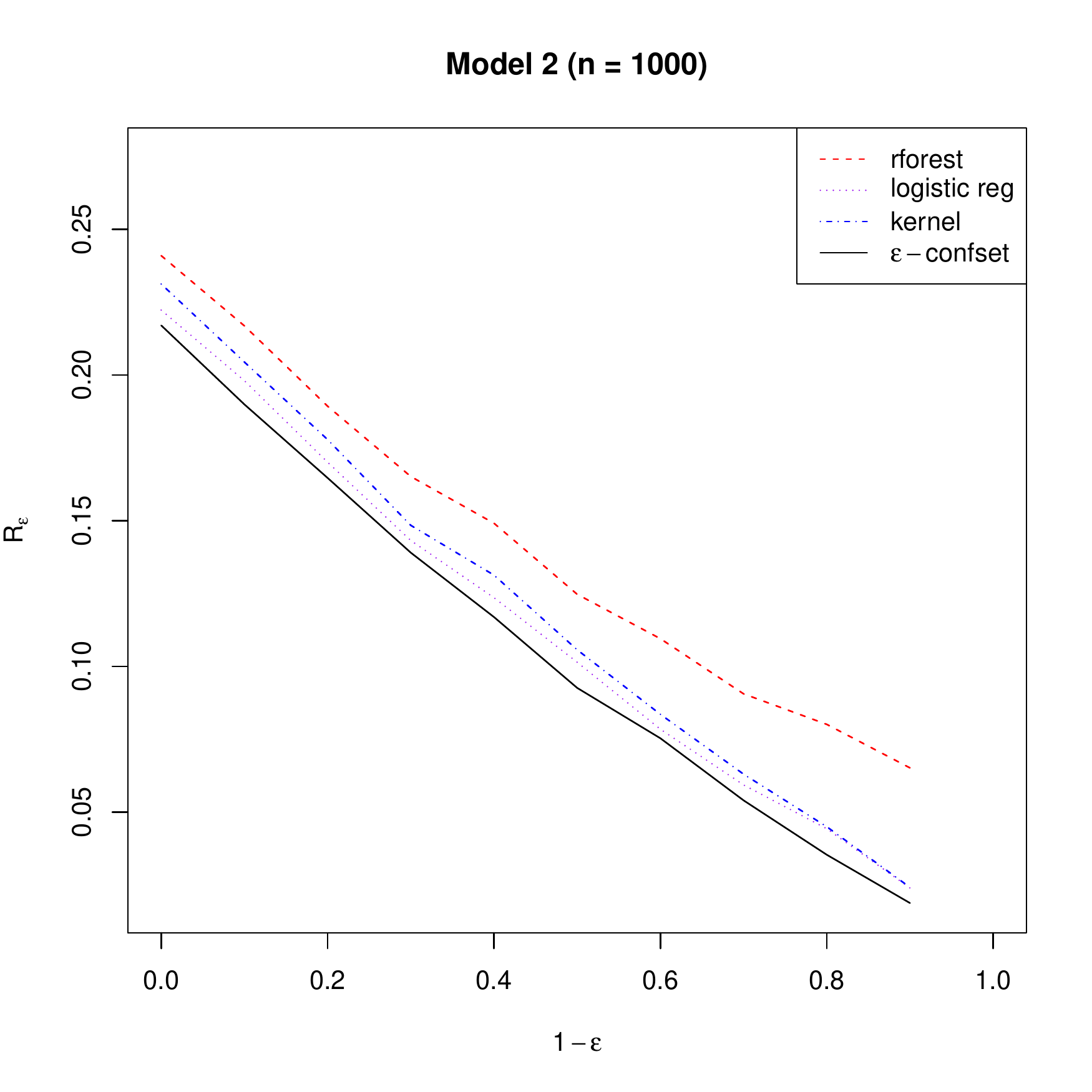}\\
\end{tabular}
\end{center}
\caption{\label{fig:results}
Visual description of the results reported in Table~\ref{table:tableb} and~\ref{table:tablep}.
For each model and each $n$, we plot, as a function of $1-\eps$, the mean ${\rm R}_{\eps}$ over the $B = 100$ 
repetitions of the estimated risks $\mathcal{R}_K$ of the $\eps$-confidence sets (solid line) and $\mathbf{R}_K$  
of the plug-in $\varepsilon$-confidence sets based on random forest (dashed line), logistic regression (dotted line) 
and kernel rule (dotted dashed line). 
Top: the data are generated according to Model 1 (left: $n = 100$; right: $n =1000$) -- 
Bottom: the data are generated according to Model 2 (left: $n = 100$; right: $n =1000$).}
\end{figure}
\begin{table}
\begin{center}
\footnotesize{
\begin{tabular}{l || ccc || ccc} 
\multicolumn{1}{c}{} & \multicolumn{6}{c}{{Model~1}}\\
\hline 
\multicolumn{1}{c}{} &  \multicolumn{3}{c}{$n = 100$} & \multicolumn{3}{c}{$n = 1000$} \\
\hline\noalign{\smallskip}
  $\varepsilon$ \;\;\;  & \;\;   \texttt{rforest} \;\;\; &  \texttt{logistic reg} \;\;\;& \texttt{kernel} & \;\;   \texttt{rforest} \;\;\; &  \texttt{logistic reg} \;\;\;& 
  \texttt{kernel} \\
\noalign{\smallskip}
\hline
\noalign{\smallskip}
 1    & 1.00 (0.00) & 1.00 (0.00) & 1.00 (0.00) \;\;&\;\;  1.00 (0.00) & 1.00 (0.00) & 1.00 (0.00) \\
 0.9  & 0.90 (0.03) & 0.90 (0.03) & 0.90 (0.04) \;\;&\;\;  0.91 (0.03) & 0.90 (0.03) & 0.90 (0.03) \\
 0.8  & 0.80 (0.04) & 0.79 (0.04) & 0.80 (0.04) \;\;&\;\;  0.81 (0.04) & 0.80 (0.04) & 0.80 (0.04) \\
 0.7  & 0.70 (0.05) & 0.69 (0.04) & 0.70 (0.04) \;\;&\;\;  0.69 (0.05) & 0.69 (0.05) & 0.69 (0.04) \\
 0.6  & 0.61 (0.05) & 0.60 (0.05) & 0.61 (0.05) \;\;&\;\;  0.60 (0.05) & 0.60 (0.05) & 0.60 (0.06) \\
 0.5  & 0.51 (0.05) & 0.49 (0.06) & 0.50 (0.06) \;\;&\;\;  0.51 (0.05) & 0.50 (0.05) & 0.51 (0.06) \\
 0.4  & 0.40 (0.05) & 0.40 (0.05) & 0.40 (0.05) \;\;&\;\;  0.40  (0.05)  &  0.40 (0.05) & 0.39  (0.05) \\
 0.3  & 0.30 (0.05) & 0.30 (0.05) & 0.30 (0.04) \;\;&\;\;  0.30 (0.05) & 0.30 (0.05) & 0.29 (0.05) \\
 0.2  & 0.20 (0.04) & 0.21 (0.05) & 0.21 (0.05) \;\;&\;\;  0.21 (0.04)  &  0.21 (0.04) &  0.20 (0.04) \\
 0.1  & 0.10 (0.03) & 0.10 (0.03) & 0.11 (0.03) \;\;&\;\;  0.11 (0.03)  &  0.10 (0.03) & 0.11  (0.03) \\
\hline
\end{tabular}
}

\vspace*{0.25cm}

\footnotesize{
\begin{tabular}{l || ccc || ccc} 
\multicolumn{1}{c}{} & \multicolumn{6}{c}{{Model~2}}\\
\hline 
\multicolumn{1}{c}{} &  \multicolumn{3}{c}{$n = 100$} & \multicolumn{3}{c}{$n = 1000$} \\
\hline\noalign{\smallskip}
  $\varepsilon$ \;\;\;  & \;\;   \texttt{rforest} \;\;\; &  \texttt{logistic reg} \;\;\;& \texttt{kernel} & \;\;   \texttt{rforest} \;\;\; &  \texttt{logistic reg} \;\;\;& 
  \texttt{kernel} \\
\noalign{\smallskip}
\hline
\noalign{\smallskip}
 1    & 1,00 (0.00)   & 1.00 (0.00)    & 1.00 (0.00)  \;\;&\;\; 1.00 (0.00) & 1.00 (0.00) & 1.00 (0.00) \\
 0.9  & 0.90 (0.04)   & 0.90 (0.03)    & 0.90 (0.04)  \;\;&\;\; 0.90 (0.03) & 0.90 (0.03) & 0.90 (0.03) \\
 0.8  & 0.81 (0.04)   & 0.81 (0.03)    & 0.81 (0.04)  \;\;&\;\; 0.80 (0.04) & 0.80 (0.04) & 0.80 (0.05) \\
 0.7  & 0.70 (0.05)   & 0.70 (0.04)    & 0.70 (0.04)  \;\;&\;\; 0.69 (0.05) & 0.69 (0.05) & 0.69 (0.05) \\
 0.6  & 0.61 (0.05)   & 0.61 (0.05)    & 0.60 (0.05)  \;\;&\;\; 0.61 (0.05) & 0.60 (0.05) & 0.60 (0.05) \\
 0.5  & 0.51 (0.05)   & 0.50 (0.04)    & 0.50 (0.04)  \;\;&\;\; 0.50 (0.05) & 0.51 (0.05) & 0.51 (0.05) \\
 0.4  & 0.39 (0.05)   & 0.40 (0.05)    & 0.39 (0.05)  \;\;&\;\; 0.40 (0.05) & 0.40 (0.05) & 0.40 (0.05) \\
 0.3  & 0.29 (0.05)   & 0.30 (0.04)    & 0.30  (0.05) \;\;&\;\; 0.30 (0.04) & 0.29 (0.05) & 0.29 (0.04) \\
 0.2  & 0.21 (0.04)   & 0.20 (0.04)    & 0.20  (0.04) \;\;&\;\; 0.20 (0.04) & 0.21 (0.04) &  0.21 (0.04) \\
 0.1  & 0.11 (0.03)   & 0.11 (0.03)    & 0.11  (0.03) \;\;&\;\; 0.12 (0.03) & 0.11 (0.03) & 0.11 (0.03) \\
\hline
\end{tabular}
}

\end{center}
\caption{\label{table:epstablep} 
For each of the $B = 100$ repetitions and each model, we derive the estimated 
proportion of classified instances $\mathcal{P}_K$ of three different plug-in $\varepsilon$-confidence sets w.r.t. $\eps$
and to the sample size $n$.
We compute the means and standard deviations (between parentheses) over the $B = 100$ repetitions. 
For each $\eps$ and each $n$, the plug-in $\varepsilon$-confidence sets are based on, from left to right, 
\texttt{rforest}, \texttt{logistic reg} and \texttt{kernel}, which are respectively the random forest, the logistic regression and the kernel rule procedures.
Top: the data are generated according to Model 1 -- Bottom: the data are generated according to Model 2.}
\end{table}
From our numerical study, we make several observations. First, as expected, the risk of the $\eps$-confidence sets is 
decreasing with $\eps$ as observed in 
Table~\ref{table:tableb}. In both models, the reject option contributes to improve the overall misclassification risk. As 
an example, we see that in Model~2 the estimated value 
of the misclassification risk, that is when $\eps$ = 1, equals $0.22$ whereas if $\eps = 0.1$ the estimated value of 
risk  is $0.02$ which is a significant improvement. 
Note that in Model~1, the classification problem is quite difficult and then the decrease of the risk seems to be 
slower and a bit less significant.
On the other hand, we also observe in Table~\ref{table:tableb} that the proportions of classified data match with the 
theoretical values.
Regarding Tables~\ref{table:tablep}-\ref{table:epstablep}, the same comments can be made in both models and whatever 
the used classification procedure.
Moreover, some features of Table~\ref{table:tablep} are worth commenting on. For fixed $\eps$ and for each scenario, 
the estimated risk of the all plug-in $\eps$-confidence 
sets decreases with $n$ which is the size of the sample used to estimate the regression function $\eta^*$. Furthermore, 
for $n = 1000$ and 
viewing Table~\ref{table:tableb}, we observe that the estimated risks of the plug-in $\eps$-confidence sets are close 
to the oracle ones. This illustrates the convergence 
result provided in Theorem~\ref{thm:mainThm}. However, we can see that the random forest procedure are outperformed 
by the other procedures (especially when $\eps$ is 
small). Indeed, the construction of plug-in $\eps$-confidence sets relies on the estimator of $\eta^*$: better estimators lead to better confidence sets.
Figure~\ref{fig:results} summarizes many aspects of our previous discussion.

\subsection{Importance of Assumptions (A1) and (A2)}
\label{subsec:assFail}

\begin{table}
\begin{center}
\footnotesize{
\begin{tabular}{ l || cc || cc} 
\multicolumn{1}{c}{} &  \multicolumn{2}{c}{\texttt{(A2) fails (CART)  }} & \multicolumn{2}{c}{\texttt{(A1) fails (kernel) }} 
 \\
\noalign{\smallskip}
\hline\noalign{\smallskip}
  $\varepsilon$ \;\;\;  & $\mathcal{P}_K$ &  $\mathbf{R}_K$ & $\mathcal{P}_K$ & $\mathbf{R}_K$   \\
\noalign{\smallskip}
\hline
\noalign{\smallskip}
 1     & \quad 1.00 (0.00)    & \quad 0.27 (0.03) \quad   & \quad 1.00 (0.00)   & \quad  0.32 (0.03)  \\
 0.9  & \quad 0.98 (0.04)    & \quad 0.27 (0.04) \quad    & \quad 0.90 (0.03)    & \quad  0.31 (0.03)  \\
 0.8  & \quad 0.90 (0.07)    & \quad 0.24 (0.03) \quad    & \quad 0.80 (0.04)    & \quad  0.29 (0.03)  \\
 0.7  & \quad 0.84 (0.10)    & \quad 0.22 (0.03) \quad    & \quad 0.70 (0.05)   & \quad  0.27 (0.04)  \\
 0.6  & \quad 0.79 (0.13)    & \quad 0.21 (0.04) \quad    & \quad 0.61 (0.05)    & \quad  0.26 (0.04)  \\
 0.5  & \quad 0.75 (0.16)    & \quad 0.21 (0.04) \quad    & \quad 0.50 (0.05)    & \quad  0.24 (0.04)  \\
 0.4  & \quad 0.60 (0.14)    & \quad 0.18 (0.05) \quad    & \quad 0.40 (0.05)    & \quad  0.23 (0.03)  \\
 0.3  & \quad 0.48 (0.13)    & \quad 0.18 (0.06) \quad    & \quad 0.30  (0.04)   & \quad  0.22 (0.03)  \\
 0.2  & \quad 0.39 (0.13)    & \quad 0.18 (0.06) \quad    & \quad 0.20  (0.04)   & \quad  0.20 (0.03)  \\
 0.1  & \quad 0.31 (0.12)    & \quad 0.16 (0.06) \quad    & \quad 0.11  (0.03)   & \quad  0.20 (0.04)  \\
\hline
\end{tabular}
}
\end{center}
\caption{\label{table:nonCont}
For each of the $B = 100$ repetitions, we derive the  estimated proportions of classified instances $\mathcal{P}_K$ and
the estimated risks $\mathbf{R}_K$ of the two plug-in $\varepsilon$-confidence sets w.r.t. $\eps$.
We compute the means and standard deviations (between parentheses) over the $B = 100$ repetitions. 
Left: the data are generated according to Model 2, then Assumption (A1) holds; the procedure used to build the 
plug-in $\eps$-confidence set is based on CART method, then {\bf Assumption (A2) fails} 
-- Right: the data are generated according to Model 3, then {\bf Assumption (A1) fails};
the procedure used to build the plug-in $\eps$-confidence set is based on kernels, then Assumption (A2) holds.}
\end{table}

In this section, we shed some light on the importance of Assumptions (A1) and (A2). More precisely, we study the behavior of 
plug-in $\eps$-confidence sets when one of these two assumptions is not satisfied.

We first consider a case where the cumulative distribution $F_{f}$ is continuous but not $F_{\hat{f}}$.
We consider the simulation scheme described in Section~\ref{subsec:a1a2} with Model~2 and parameters $n = 100$, $N = 100$ and  $K = 1000$.
But this time, the plug-in $\eps$-confidence set relies on the CART procedure which involves that the Assumption (A2) does not hold. 
The obtained results are reported in Table~\ref{table:nonCont}--Left.
Two observations can be made. First, judging by the estimated proportions of classified instances and by the associated standard deviations, we are not able to 
control these proportions. Therefore, one of the important feature of our procedure fails. Second, 
although the risk of misclassification is decreasing with $\eps$, this decrease is quite slow and cannot be as important as observed
with plug-in confidence sets studied in Section~\ref{subsec:a1a2} (see Table~\ref{table:tablep}). Indeed, for CART method and more generally if
Assumption (A2) does not hold, the proportion of rejected data is usually not large enough.   

Next, we study the reverse case where the cumulative distribution function $F_{\hat{f}}$ is continuous but not $F_f$.
We consider the following model.
\begin{itemize}
\item Model~3:
\begin{enumerate}
\item the feature $X  \overset{\mathcal{L}}{=} U$, where $U$ follows a uniform distribution on $[0,1]$;
\item conditional on $X$,  the label $Y$ is drawn according a Bernoulli distribution with parameter 
\begin{equation*}
\eta^*(X) = \frac{1}{5} \1_{\{X \leq 1/4\}} + \frac{2}{5} \1_{\{1/4 < X \leq 1/2\}} + \frac{3}{5} \1_{\{1/2 < \leq 3/4\}} + \frac{4}{5} \1_{\{3/4 < X\}}. 
\end{equation*}
\end{enumerate}
\end{itemize}
Then, for this model, $F_{f}$ is not continuous. Moreover,
we have that, for $\eps \in [0.5,1]$, $\mathcal{P}(\Gamma^{\bullet}_{\eps}) = 1$ and $\mathcal{R}(\Gamma^{\bullet}_{\eps}) = 3/10$,
and for $\eps \in ]0,0.5[$, $\mathcal{P}(\Gamma^{\bullet}_{\eps}) = 1/2$ and $\mathcal{R}(\Gamma^{\bullet}_{\eps}) = 1/5$.
For this model, we use the simulation scheme described in Section~\ref{subsec:a1a2} with
the plug-in $\eps$ confidence sets which relies on the kernel rule and the samples sizes $n= 1000$, $N = 100$ and $K = 1000$.
The results are provided in Table~\ref{table:nonCont}--Right.
As a remark, we first note that since Assumption (A2) is satisfied the proportions of classified instances match with the
theoretical values. Second, the estimated risk of misclassification of the plug-in $\eps$-confidence set decreases with $\eps$.
However, from our point of view, it is irrelevant to compare the performances of the $\eps$-confidence sets and those of the plug-in $\eps$-confidence sets.
Indeed, except for $\eps = 1$, the proportions of classified data differ. As an example, if $\eps = 0.7$, the estimated risk of the plug-in $\eps$-confidence set
is equal to $0.27$ which seems better than the risk of the $\eps$-confidence set. But, for $\eps = 0.7$ the proportion of the classified instances is larger for
the $\eps$-confidence set and equals~1.

\section{Conclusion}
\label{sec:conclusion}
In the classification with reject option framework, 
we introduce a new procedure that allows us to control exactly the rejection probability. 
The construction of the $\eps$-confidence sets and their plug-in approximations
relies on the cumulative distribution function of the score functions $f^*$ and $\hat{f}$.
Theoretical guarantees, especially rates of convergence, involve the continuity of these cumulative distribution
functions. Numerical experiments emphasize the importance of the continuity assumption.
As viewed in Section~\ref{sec:Estimateion}, the plug-in $\eps$-confidence set
is defined as a two steps algorithm whose second step consists in 
the estimation of the cumulative distribution function $F_{\hat{f}}$. 
Interestingly, this step does not require a set of labeled data that is suitable for semi-supervised learning.
In a future work, we intent to generalize our procedure to the multiclass case and
study procedures based on empirical risk minimization.

\section{Appendix}
\label{sec:Appendix}
This section gathers the proofs of our results.

\subsection{Proof of Proposition~\ref{prop:prop2}}

We first define the following events
\begin{eqnarray*}
\mathcal{A}_y & = & \{f^*(\Xp) \geq \alpha_{\varepsilon},  \card(\Gamma_s(\Xp)) \neq 1, s^{*}(\Xp) \neq y\},\;\; y= 0,1, \\
\mathcal{B}_y & = & \{f^*(\Xp) < \alpha_{\varepsilon},  \card(\Gamma_s(\Xp)) = 1,  s(\Xp) \neq y\},\;\; y = 0,1,\\
\mathcal{C} & = &  \{f(^*\Xp) \geq \alpha_{\varepsilon},  \card\left(\Gamma_s(\Xp)\right) = 1,  s^{*}(\Xp) \neq s(\Xp)\},\\
\mathcal{D} & = & \mathcal{A}_0 \cup \mathcal{A}_1 \cup \mathcal{B}_0 \cup \mathcal{B}_1,
\end{eqnarray*}
and the random variable
\begin{equation*}
U = \mathbf{1}_{\{s(\Xp) \neq \Yp, \card(\Gamma_s(\Xp)) = 1\}} - \mathbf{1}_{\{s^{*}(\Xp) \neq \Yp, \card(\Gamma_{\varepsilon}^{\bullet}(\Xp) )= 1\}}.
\end{equation*}
Since $\mathcal{P}(\Gamma_s) = \mathcal{P}(\Gamma_{\eps}^{\bullet}) = \eps$, the proof of the proposition relies on the decomposition of the conditional expectation of $U$ 
given $\Xp$ over the sets $\mathcal{C}$ and $\mathcal{D}$.
We have
\begin{multline*}
 \mathbb{E}\left[U \mathbf{1}_\mathcal{C}|\Xp\right] =  \mathbf{1}_\mathcal{C} \{\eta^*(\Xp)\mathbf{1}_{\{s^{*}(\Xp) = 1\}} - \eta^*(\Xp)\mathbf{1}_{\{s^{*}(\Xp) = 0\}} +\\ 
 (\eta^*(\Xp) - 1)\mathbf{1}_{\{s^{*}(\Xp)  =  1\}} + (1-\eta^*(\Xp))\mathbf{1}_{\{s^{*}(\Xp) = 0\}}\}.
\end{multline*}
Since, $s^{*}(\Xp) = 1$ and $s^{*}(\Xp) = 0$ imply respectively that $  \eta^*(\Xp) \geq 1/2$ and $ \eta^*(\Xp) \leq 1/2$, we obtain from the above decomposition
\begin{equation}
\label{eqproofintermC}
 \mathbb{E}\left[U \mathbf{1}_\mathcal{C}\right] = \mathbb{E}\left[|2\eta^*(\Xp) -1 | \mathbf{1}_\mathcal{C}\right].
\end{equation}

Next,
\begin{equation}
\label{eq:eqproof1}
\mathbb{E}\left[U \mathbf{1}_\mathcal{D}|\Xp\right] = \eta^*(\Xp)\mathbf{1}_{\mathcal{B}_1} + (1- \eta^*(\Xp))\mathbf{1}_{\mathcal{B}_0} - \eta^*(\Xp)\mathbf{1}_{\mathcal{A}_1} - 
(1- \eta^*(\Xp))\mathbf{1}_{\mathcal{A}_0}
\end{equation}
Since, $\mathcal{P}(\Gamma_s) = \mathcal{P}(\Gamma_{\varepsilon}^{\bullet}) =  \mathbb{P}\left(f^*( \Xp) \geq \alpha_{\varepsilon}  
\right)$, we deduce
\begin{equation*} 
\mathbb{P}(\card(\Gamma_s(\Xp) )= 1, f^*(\Xp) < \alpha_{\varepsilon}) = \mathbb{P}\left(\card(\Gamma_s(\Xp)) \neq 1, f^*(\Xp) \geq \alpha_{\varepsilon}\right),
\end{equation*}
which implies 
\begin{equation*}
 (1-\alpha_{\eps})\mathbb{E}\left[\1_{\mathcal{B}_0 \cup \mathcal{B}_1} \right] - (1-\alpha_{\eps})\mathbb{E}\left[\1_{\mathcal{A}_0\cup \mathcal{A}_1}\right]  = 0.
\end{equation*}
Therefore, adding this null term to~\eqref{eq:eqproof1}, we obtain
\begin{multline}
\label{eq:eqproof2}
\mathbb{E}\left[U \mathbf{1}_{\mathcal{D}}\right] = \mathbb{E}\left[(\alpha_{\eps} - (1-\eta^*(\Xp))\mathbf{1}_{\mathcal{B}_1} +  (\alpha_{\eps} - 
\eta^*(\Xp))\mathbf{1}_{\mathcal{B}_0} \right] \\ 
 +\mathbb{E}\left[(\eta^*(\Xp) - \alpha_{\eps})\mathbf{1}_{\mathcal{A}_0} + ((1- \eta^*(\Xp) ) - \alpha_{\eps})\mathbf{1}_{\mathcal{A}_1}\right].
\end{multline}
Note that, 
\begin{eqnarray*}
 f^*(\Xp) < \alpha_{\eps} \qquad\qquad\qquad \quad  \ \, & \Rightarrow  & (\alpha_{\eps} - (1 -  \eta^*(\Xp)) \geq  0   \;\;  {\rm and} \;\; (\alpha_{\eps} - \eta^*(\Xp)) \geq 
 0\\
f^*(\Xp) \geq \alpha_{\eps} \;\; {\rm and} \;\;  s^{*}(\Xp) \neq 1 & \Rightarrow & (1- \eta^*(\Xp) - \alpha_{\eps}) \geq 0\\
f^*(\Xp) \geq \alpha_{\eps} \;\; {\rm and} \;\;  s^{*}(\Xp) \neq 0 & \Rightarrow & (\eta^*(\Xp) - \alpha_{\eps}) \geq 0.
\end{eqnarray*}
Hence, from~\eqref{eq:eqproof2}, we can write
\begin{equation*}
\mathbb{E}\left[U \mathbf{1}_{\mathcal{D}}\right] = \mathbb{E}\left[|\eta^*(\Xp) - \alpha_{\eps}| \1_{\mathcal{A}_0 \cup \mathcal{B}_0}\right] + 
\mathbb{E}\left[|1 - \eta^*(\Xp) - \alpha_{\eps}| \1_{\mathcal{A}_1 \cup \mathcal{B}_1}\right].
\end{equation*}
Combining this result with~\eqref{eqproofintermC} shows that $ \mathbb{E}\left[U \mathbf{1}_{\mathcal{C} \cup  \mathcal{D}}\right] \geq 0 $, and provides in the same time 
the desired result.


\subsection{Proof of Proposition~\ref{prop:prop3}}

We first prove the following inequality for $\alpha,\tilde{\alpha} \in [0,1/2[$, $\alpha \leq \tilde{\alpha}$
\begin{equation}
\label{eq:ineqE1}
\mathbb{P}(s^{*}(\Xp) \neq \Yp) | f^*(\Xp) \geq \tilde{\alpha}) \leq \mathbb{P}(s^{*}(\Xp) \neq \Yp) | f^*(\Xp) \geq \alpha),
\end{equation}
Since for $\eps,\eps^{'} \in ]0,1]$ one has $\eps \leq \eps^{'} \Leftrightarrow \alpha_{\eps} \geq \alpha_{\eps^{'}}$,
a direct application of~\eqref{eq:ineqE1} yields the proposition.
In order to prove~\eqref{eq:ineqE1}, recall that
\begin{equation}
\label{eq:eqUtile}
\dfrac{x}{y} - \dfrac{z}{t} = \dfrac{1}{2yt} \left((x-z)(t+y) + (x+z)(t-y)\right), \quad \forall x,z \in \mathbb{R}, \,\,\forall y,t \in \mathbb{R}\setminus \{0 \} .
\end{equation}
Thus, if we define
\begin{eqnarray*}
C_1 & = & \mathbb{P}(s^{*}(\Xp) \neq \Yp, f^*(\Xp) \geq \alpha) - \mathbb{P}(s^{*}(\Xp) \neq \Yp, f^*(\Xp) \geq \tilde{\alpha})\\
C_2 & = & \mathbb{P}(f^*(\Xp) \geq \alpha) + \mathbb{P}(f^*(\Xp) \geq \tilde{\alpha})\\
C_3 & = & \mathbb{P}(s^{*}(\Xp) \neq \Yp, f^*(\Xp) \geq \alpha) + \mathbb{P}(s^{*}(\Xp) \neq \Yp, f^*(\Xp) \geq \tilde{\alpha})\\
C_4 & = & \mathbb{P}(f^*(\Xp) \geq \tilde{\alpha}) - \mathbb{P}(f^*(\Xp) \geq \alpha),
\end{eqnarray*} 
from~\eqref{eq:eqUtile}, we have
\begin{equation*}
\mathbb{P}(s^{*}(\Xp) \neq \Yp | f^*(\Xp) \geq \tilde{\alpha})  \leq  \mathbb{P}(s^{*}(\Xp) \neq \Yp | f^*(\Xp) \geq \alpha)  
\Leftrightarrow 
C_1C_2+C_3C_4 \geq 0.
\end{equation*}
Since $\mathbb{P}(s^{*}(\Xp) \neq \Yp | \Xp) = 1-f^*(\Xp)$, we deduce that
\begin{multline*}
C_1 C_2 + C_3C_4  = 
\mathbb{E}\left[(1-f^*(\Xp)) \1_{\{\alpha \leq f^*(\Xp) \leq \tilde{\alpha}\}}\right] \mathbb{E}\left[\1_{\{f^*(\Xp) \geq \alpha\}}
+ \1_{\{f^*(\Xp) \geq \tilde{\alpha}\}}\right]\\
-\mathbb{E}\left[(1-f^*(\Xp))\left(\1_{\{f^*(\Xp) \geq \alpha\}} + \1_{\{f^*(\Xp) \geq \tilde{\alpha}\}}\right)\right]
\mathbb{E}\left[\1_{\{\alpha \leq f^*(\Xp) \leq \tilde{\alpha}\}}\right].
\end{multline*}
Note that 
$\mathbb{E}\left[\1_{\{f^*(\Xp) \geq \alpha\}}\right] = \mathbb{E}\left[\1_{\{\alpha \leq f^*(\Xp) \leq \tilde{\alpha}\}}\right] + \mathbb{E}\left[\1_{\{f^*(\Xp) \geq 
\tilde{\alpha}\}}\right]$.\\ 
Hence, from the above decomposition, we obtain
\begin{multline*}
C_1 C_2 + C_3 C_4  = 2\mathbb{E}\left[\1_{\{\alpha \leq f^*(\Xp) \leq \tilde{\alpha}\}}\right]
\mathbb{E}\left[f^*(\Xp) \1_{\{f^*(\Xp) \geq \tilde{\alpha}\}}\right]\\
- 2 \mathbb{E}\left[\1_{\{f^*(\Xp) \geq \tilde{\alpha}\}}\right] \mathbb{E}\left[f^*(\Xp)\1_{\{\alpha \leq f^*(\Xp) \leq 
\tilde{\alpha}\}}\right].
\end{multline*}
Since,
\begin{eqnarray*}
\mathbb{E}\left[\1_{\{\alpha \leq f^*(\Xp) \leq \tilde{\alpha}\}} \right]
\mathbb{E}\left[f^*  (\Xp) \1_{\{f^*(\Xp) \geq \tilde{\alpha}\}}\right] &\geq & 
\tilde{\alpha} \mathbb{P}\left(\alpha \leq f^*(\Xp) \leq \tilde{\alpha}\right) \mathbb{P}(f^*(\Xp) \geq \tilde{\alpha}) \;\; {\rm and} 
\\
\mathbb{E}\left[\1_{\{f^*(\Xp) \geq \tilde{\alpha}\}}\right]\mathbb{E}\left[f^* (\Xp)\1_{\{\alpha \leq f^*(\Xp) \leq \tilde{\alpha}\}}\right] & \leq & 
\tilde{\alpha} \mathbb{P}(\alpha \leq f^*(\Xp) \leq \tilde{\alpha}) \mathbb{P}(f^*(\Xp) \geq \tilde{\alpha}), \\
\end{eqnarray*}
we deduce Inequality~\eqref{eq:ineqE1}.


\subsection{Proof of Proposition~\ref{Th:RiskBayes}}

This section is devoted to the proof of the result related to the Gaussian mixture model.
Before starting, let us state a few properties that will be often used. \\
Let us write for short $f_1$ and $f_0$ instead of $\eta^*$ and $1-\eta^*$ respectively, so that $f=\max \{ f_0 , f_1 \}$.
Hence, we can write:
$$
\eta^*(x) = f_1(x) =  \mathbb{P}(Y=1 | X=x) =   \frac{ \mathbb{P}(X=x | Y=1) }{ \mathbb{P}(X=x | Y=1) + \mathbb{P}(X=x | Y=0)   } :=   \frac{ p_1 (x) }{ p_1 (x) +p_0 (x) },
$$
for any $x\in \mathcal{X}$.
Then, for $y=0,1$, using the fact that given $Y_\bullet = y$, the random variable $ X_\bullet^\top  \Sigma^{-1}  (  \mu_0-  \mu_1  ) \sim \mathcal{N} \left( \mu_y^\top   \Sigma^{-1}  (  \mu_0-  \mu_1  )   \ , \    \Vert   \mu_1-  \mu_0    \Vert_{ \Sigma^{-1}}^2  \right) $ we get, on the event $\{Y_\bullet = y\}$
\begin{eqnarray}
	\label{eq:SneY1}
		s^*(X_\bullet)  \neq Y_\bullet
		&\Leftrightarrow &
		f^*(X_\bullet) = f_{1-y}(X_\bullet)
		\quad 
		\Leftrightarrow 
		\quad
		f_y(X_\bullet) \leq \frac{1}{2}
		\quad
		\Leftrightarrow
		 \quad
		\log \left(  \frac{p_{1-y}(X_\bullet)}{p_y(X_\bullet)}\right) \geq 0
		\nonumber
		\\
		&\Leftrightarrow &
		X_\bullet ^\top  \Sigma^{-1}  (  \mu_{1-y}-  \mu_y  )    -   \frac{1}{2}     \mu_{1-y}^\top   \Sigma^{-1}     \mu_{1-y}    +    \frac{1}{2}  \mu_y^\top   \Sigma^{-1}     \mu_y    \geq 0 
		\nonumber 
		\\
		&\Leftrightarrow &
		\left( X_\bullet - \mu_y   \right)^\top  \Sigma^{-1} (  \mu_{1-y}-  \mu_y  )    -   \frac{1}{2}     \Vert   \mu_1-  \mu_0    \Vert_{ \Sigma^{-1}}^2 \geq 0 
		\nonumber 
		\\
		&\Leftrightarrow &
		\frac{\left( X_\bullet - \mu_y   \right)^\top  \Sigma^{-1}  (  \mu_{1-y}-  \mu_y  ) }{   \Vert   \mu_1-  \mu_0    \Vert_{ \Sigma^{-1}}     }      -   \frac{1}{2}     \Vert   \mu_1-  \mu_0    \Vert_{ \Sigma^{-1}}   \geq 0 .
\end{eqnarray}
where $\Vert \cdot \Vert_{\Sigma^{-1}} $ denotes the norm under $ \Sigma^{-1} $: $\Vert \mu \Vert_{\Sigma^{-1}} ^2 = \mu^\top \Sigma^{-1}  \mu$, for any $\mu \in \mathcal{X}$.

\subsubsection{Intermediate results}
The proof of Proposition~\ref{Th:RiskBayes} relies on two intermediate results. Then we state them first and give their proofs. They bring into play the cumulative distribution $F_f^*$.


\begin{proposition}\label{prop:interm1} 
	Let $y\in \{0,1\}$. Conditional on the event $Y_\bullet=y$ we have
\begin{multline*}
		F_f^*( f_{1-y}(X_\bullet)) = \Phi \left(    \frac{    \left( X_\bullet - \mu_y\right)  ^\top  \Sigma^{-1}  (  \mu_{1-y}-  \mu_y ) }{    \Vert \mu_1  -    \mu_0 \Vert_{ \Sigma^{-1}}  }      \right)      + 
\\
		    \Phi \left(      \frac{    \left( X_\bullet - \mu_y\right)  ^\top  \Sigma^{-1}  (  \mu_{1-y}-  \mu_y  ) }{    \Vert \mu_1  -    \mu_0 \Vert_{ \Sigma^{-1}}  }  -  \Vert   \mu_1-  \mu_0 \Vert_{ \Sigma^{-1}}      \right)    - 1,
\end{multline*}
where $\Phi$ is the standard normal cumulative distribution function.
\end{proposition}
\begin{proof}
To prove this result, we need to investigate the function $F_f^*(\cdot) = \mathbb{P}(f^*(X) \leq \cdot)$.
Let $\alpha\in[1/2,1]$. We have
\begin{multline}
 \mathbb{P}(f^*(X) \leq \alpha)  =  
\mathbb{P}(f^*(X) \leq \alpha , f_1(X) \geq f_0(X) ) +  \mathbb{P}(f^*(X) \leq \alpha , f_1(X) \leq f_0(X) ) =  
 \\ 
\frac{1}{2}\mathbb{P}(f_1(X) \leq \alpha , f_1(X) \geq f_0(X) | Y=1) + \frac{1}{2}  \mathbb{P}(f_1(X) \leq \alpha , f_1(X) \geq f_0(X) | Y=0 ) + 
 \\
\frac{1}{2}\mathbb{P}(f_0(X) \leq \alpha , f_1(X) \leq f_0(X) | Y=1) + \frac{1}{2}  \mathbb{P}(f_0(X) \leq \alpha , f_1(X) \leq f_0(X) | Y=0 ) ,
\label{eq:AnnStuGEtap1}
\end{multline}
where we used in the last equality the fact that $Y$ is a Bernoulli random variable with parameter $1/2$. As already seen, we have for $y\in \{0,1\}$,
$$
f_y(x) = \mathbb{P}\left( Y=y | X=x \right) =  \frac{ p_y (x)}{ p_1 (x)+p_0 (x)}.
$$
Hence, denoting by $u$ the function from $[1/2, 1)$ into $[1, +\infty)$ defined by $ u(\alpha) = \frac{\alpha}{1-\alpha}$ and fixing this notation in the above relation~\eqref{eq:AnnStuGEtap1}, we get
\begin{eqnarray}
\label{eq:proofG}
\mathbb{P}(f^*(X) \leq \alpha) = & \frac{1}{2}& \left[ 
 \mathbb{P}\left(\frac{p_1 (X)}{p_0(X)}   \in  [1,u(\alpha) ] \  | \  Y=1\right) +  \mathbb{P}\left(\frac{p_1(X)}{p_0(X)}   \in  [1,u(\alpha) ]  \  |  \ Y=0 \right)
\right.   \nonumber  \\
&&+
\left.
\mathbb{P}\left(\frac{p_0(X)}{p_1(X)}   \in  [1,u(\alpha) ] \  | \  Y=1\right) +  \mathbb{P}\left(\frac{p_0(X)}{p_1(X)}   \in  [1,u(\alpha) ]  \  |  \ Y=0 \right)
\right]    \nonumber  \\
:=& & \frac{1}{2} (A_1 + A_2 + A_3 + A_4).
\end{eqnarray}
All of the terms $A_1, A_2, A_3, A_4$ will be treated in the same way. Then let us consider $A_1$ for instance: using 
very close reasoning as in~\eqref{eq:SneY1} with $y=1$, we have
\begin{eqnarray*}
A_1  &  =  &  \mathbb{P}\left( 0 \leq  \log \left( \frac{p_1(X)}{p_0(X)} \right) \leq   \log \left(u(\alpha)  \right)  
\  | \  Y=1\right)  
\\
& = & 
  \mathbb{P}\left( 0 \leq  
- \left( X - \mu_1\right)^\top  \Sigma^{-1}  (  \mu_0-  \mu_1  )    + \frac{1}{2}     \Vert   \mu_1-  \mu_0    \Vert_{ 
\Sigma^{-1}}^2
  \leq   \log \left(u(\alpha)  \right)  \  | \  Y=1\right)  
\\
& = & 
  \mathbb{P}\left( 0 \leq  
Z + \frac{1}{2}     \Vert   \mu_1-  \mu_0    \Vert_{ \Sigma^{-1}}
 \leq 
 \frac {\log \left(u(\alpha)  \right)}{ \Vert   \mu_1-  \mu_0    \Vert_{ \Sigma^{-1}}}      \right)  
\\
& = & 
  \mathbb{P}\left( -  \frac{1}{2}     \Vert   \mu_1-  \mu_0    \Vert_{ \Sigma^{-1}}      \leq  Z  \leq   \frac {\log \left(u(\alpha) \right)}{ \Vert   \mu_1-  \mu_0    \Vert_{ \Sigma^{-1}}}     -  \frac{1}{2}     \Vert   \mu_1-  \mu_0    \Vert_{ \Sigma^{-1}}   \right),
\end{eqnarray*}
where $Z$ is normally distributed. In the same way, we get
\begin{eqnarray*}
A_1 = A_4 & = &  \mathbb{P}\left(  - \frac{  \Vert    \mu_1-  \mu_0   \Vert_{ \Sigma^{-1}} }{2}    \leq Z  \leq  \frac {\log \left(u(\alpha)  \right)}{ \Vert    \mu_1-  \mu_0    \Vert_{ \Sigma^{-1}}}    -  \frac{  \Vert    \mu_1-  \mu_0   \Vert_{ \Sigma^{-1}} }{2}   \right)
\\
A_2 = A_3 &  =  &   \mathbb{P}\left(  \frac{  \Vert    \mu_1-  \mu_0   \Vert_{ \Sigma^{-1}} }{2}     \leq Z  \leq \frac {\log \left(u(\alpha)  \right)}{ \Vert    \mu_1-  \mu_0 \Vert_{ \Sigma^{-1}}}    +  \frac{  \Vert   \mu_1-  \mu_0    \Vert_{ \Sigma^{-1}} }{2}   \right).
\end{eqnarray*}
Coming back to~\eqref{eq:proofG} and using twice the following relation $ \Phi( x)  +  \Phi( -x) = 1$ for any $x\in \mathbb{R}$ which is valid for the normal distribution, we easily get
\begin{eqnarray}
	\label{eq:AppAIntroG}
&& F_f^*(\alpha) =  \mathbb{P}(f^*(X) \leq \alpha) \nonumber \\
& = &  
\mathbb{P}\left(  - \frac{  \Vert    \mu_1-  \mu_0    \Vert_{ \Sigma^{-1}} }{2}    \leq Z  \leq  \frac {\log \left(u(\alpha)  \right)}{ \Vert    \mu_1-  \mu_0    \Vert_{ \Sigma^{-1}}}    -  \frac{  \Vert    \mu_1-  \mu_0   \Vert_{ \Sigma^{-1}} }{2}   \right)
 \nonumber
\\
&& 
+
 \mathbb{P}\left(  \frac{  \Vert   \mu_1-  \mu_0    \Vert_{ \Sigma^{-1}} }{2}     \leq Z  \leq \frac {\log \left(u(\alpha)  \right)}{ \Vert    \mu_1-  \mu_0   \Vert_{ \Sigma^{-1}}}    +  \frac{  \Vert     \mu_1-  \mu_0   \Vert_{ \Sigma^{-1}} }{2}   \right) 
\nonumber  \\
& = & 
\Phi \left(  \frac {\log \left(u(\alpha)  \right)}{ \Vert   \mu_1-  \mu_0   \Vert_{ \Sigma^{-1}}}    -  \frac{  \Vert    \mu_1-  \mu_0  \Vert_{ \Sigma^{-1}} }{2}   \right) 
+
\Phi \left(  \frac {\log \left(u(\alpha)  \right)}{ \Vert   \mu_1-  \mu_0  \Vert_{ \Sigma^{-1}}}    +  \frac{  \Vert   \mu_1-  \mu_0   \Vert_{ \Sigma^{-1}} }{2}   \right)  
\nonumber \\
&&
- 
\left( 
\Phi \left(   - \frac{  \Vert   \mu_1-  \mu_0   \Vert_{ \Sigma^{-1}} }{2}     \right)   
+\Phi \left(   \frac{  \Vert    \mu_1-  \mu_0  \Vert_{ \Sigma^{-1}} }{2}    \right)
\right)
\nonumber \\
&=&
 \Phi \left(  \frac{  \Vert     \mu_1-  \mu_0   \Vert_{ \Sigma^{-1}} }{2}   +    \frac {\log \left(u(\alpha)  \right)}{ \Vert     \mu_1-  \mu_0  \Vert_{ \Sigma^{-1}}}     \right)  
-
\Phi \left(   \frac{  \Vert    \mu_1-  \mu_0   \Vert_{ \Sigma^{-1}} }{2}   - \frac {\log \left(u(\alpha)  \right)}{ \Vert     \mu_1-  \mu_0   \Vert_{ \Sigma^{-1}}}     \right) 
\nonumber \\
& =& 
\mathbb{P}\left( Z \in \left[     \frac{  \Vert    \mu_1-  \mu_0  \Vert_{ \Sigma^{-1}} }{2}   - \frac {\log 
\left(u(\alpha) \right)}{ \Vert    \mu_1-  \mu_0   \Vert_{ \Sigma^{-1}}}   \  ,  \  \frac{  \Vert   \mu_1-  \mu_0 
\Vert_{ \Sigma^{-1}} }{2}   +    \frac {\log \left(u(\alpha)  \right)}{ \Vert    \mu_1-  \mu_0   \Vert_{ \Sigma^{-1}}}   
\right]     \right) .
\end{eqnarray}
At this point, we are ready to evaluate the quantity $F_f^*( f_{1-y}(X_\bullet))$ on the event $\{Y_\bullet=y\}$ with 
$y\in \{0,1\}$. Indeed, according to~\eqref{eq:AppAIntroG}, we only need to evaluate $\frac {\log \left(u(\alpha)  
\right)}{ \Vert     \mu_1-  \mu_0   \Vert_{ \Sigma^{-1}}}  $ for $\alpha = f_{1-y}(X_\bullet)$. Thanks 
to~\eqref{eq:SneY1} we can write when  $Y_\bullet=y$
$$
\frac{ \log(u(f_{1-y}(X_\bullet))) }{\Vert \mu_1  -    \mu_0 \Vert_{ \Sigma^{-1}}    }
= 
\frac{  \log \left(\frac{p_{1-y}(X_\bullet)}{p_{y}(X_\bullet)} \right) }{ \Vert \mu_1  -    \mu_0 \Vert_{ \Sigma^{-1}}    
}
= 
\frac{  \left( X_\bullet - \mu_y\right)  ^\top  \Sigma^{-1}  (  \mu_{1-y}-  \mu_y  ) }{  \Vert \mu_1  -    \mu_0 
\Vert_{ \Sigma^{-1}}   }   -   \frac{1}{2}   \Vert \mu_1  -    \mu_0 \Vert_{ \Sigma^{-1}}  .
$$
Finally, using~\eqref{eq:AppAIntroG}, we get when $Y_\bullet=y$,  
\begin{multline*}
		F_f^*( f_{1-y}(X_\bullet)) =
 \\
		\mathbb{P}\left( Z \in \left[   \Vert   \mu_1-  \mu_0 \Vert_{ \Sigma^{-1}}   -  \frac{    \left( X_\bullet - \mu_y\right)  ^\top  \Sigma^{-1}  (  \mu_{1-y}-  \mu_y  ) }{    \Vert \mu_1  -    \mu_0 \Vert_{ \Sigma^{-1}}  }      \  ,  \        \frac{    \left( X_\bullet - \mu_y\right)  ^\top  \Sigma^{-1}  (  \mu_{1-y}-  \mu_y  ) }{    \Vert \mu_1  -    \mu_0 \Vert_{ \Sigma^{-1}}  }       \right]     \right) =
 \\
		\Phi \left(    \frac{    \left( X_\bullet - \mu_y\right)  ^\top  \Sigma^{-1}  (  \mu_{1-y}-  \mu_y  ) }{    \Vert \mu_1  -    \mu_0 \Vert_{ \Sigma^{-1}}  }      \right)     -       \Phi\left(   \Vert   \mu_1-  \mu_0 \Vert_{ \Sigma^{-1}}    -    \frac{    \left( X_\bullet - \mu_y\right)  ^\top  \Sigma^{-1}  (  \mu_{1-y}-  \mu_y  ) }{    \Vert \mu_1  -    \mu_0 \Vert_{ \Sigma^{-1}}  }      \right) = 
 \\
		\Phi \left(    \frac{    \left( X_\bullet - \mu_y\right)  ^\top  \Sigma^{-1}  (  \mu_{1-y}-  \mu_y  ) }{    \Vert \mu_1  -    \mu_0 \Vert_{ \Sigma^{-1}}  }      \right)      +     \Phi\left(      \frac{    \left( X_\bullet - \mu_y\right)  ^\top  \Sigma^{-1}  (  \mu_{1-y}-  \mu_y  ) }{    \Vert \mu_1  -    \mu_0 \Vert_{ \Sigma^{-1}}  }  -  \Vert   \mu_1-  \mu_0 \Vert_{ \Sigma^{-1}}      \right)    - 1,
\end{multline*}
where we have also used the relation $\Phi(x)+\Phi(-x) = 1$, for $x\in \mathbb{R}$ in the last line, since $\Phi$ is the normal cumulative distribution function. This ends the proof.
\end{proof}

The next result is the key tool in the proof of Proposition~\ref{Th:RiskBayes}.

\begin{proposition}\label{prop:bayesProba} 
Let $\eps \in ]0,1]$. For $y\in\{ 0, 1\}$, we have 
\begin{eqnarray*}
&&
\mathbb{P}( s^*(X_\bullet)  \ne  Y_\bullet \ , \ F_f^* \left( f^*(X_\bullet) \right)   \geq 1-\varepsilon \  | \  Y_\bullet = y )  
\\
& = &
\mathbb{P}  \left(    \left\{   \Phi \left(Z  \right)     +      \Phi \left( Z  -   \Vert \mu_1  -    \mu_0 \Vert_{ \Sigma^{-1}}   \right)         \right\}  \geq 2 -  \varepsilon    \ , \  Z   \geq  \frac{   \Vert \mu_1  -    \mu_0 \Vert_{ \Sigma^{-1}}   } {2}    \right),
\end{eqnarray*}
where $Z \sim \mathcal{N}   \left(  0,1 \right)    $.
\end{proposition}

\begin{proof} Let $\eps \in ]0,1]$. 
For $y\in \{0,1\}$, according to the first equivalence stated in~\eqref{eq:SneY1}, we observe that
\begin{multline}
\label{eq:transmdfz}
\mathbb{P}\left(      F_f^* \left( f^*(X_\bullet) \right)   \geq 1-\varepsilon              \ , \       s^*(X_\bullet)  
\ne  Y_\bullet  \  | \  Y_\bullet = y \right)  =
\\
\mathbb{P}\left(      F_f^* \left( f_{1-y}(X_\bullet) \right)   \geq 1-\varepsilon              \ , \       f^*(X_\bullet) 
= f_{1-y}(X_\bullet)  \  | \  Y_\bullet = y \right).
\end{multline}
Moreover, using the last equivalence in~\eqref{eq:SneY1}, we have, if $Y_\bullet = y $
\begin{equation}
\label{eq:reecrits}
		 f^*(X_\bullet) = f_{1-y}(X_\bullet) 
		\quad
		\Leftrightarrow
		 \quad
		\frac{\left( X_\bullet - \mu_y   \right)^\top  \Sigma^{-1}  (  \mu_{1-y}-  \mu_y  )}{   \Vert   \mu_1-  \mu_0    \Vert_{ \Sigma^{-1}}     }      -   \frac{1}{2}     \Vert   \mu_1-  \mu_0    \Vert_{ \Sigma^{-1}}   \geq 0 .
\end{equation}
Then we just need to rewrite the event $ \left\{F_f^* \left( f_{1-y}(X_\bullet) \right)   \geq 1-\varepsilon  \right\}$, when conditioned on the event $\{Y_\bullet = y \}$, in a convenient way. 
Using Proposition~\ref{prop:interm1}, we can write that when $Y_\bullet = y $
\begin{multline}
		F_f^*( f_{1-y}(X_\bullet)) = \Phi \left(    \frac{    \left( X_\bullet - \mu_y\right)  ^\top  \Sigma^{-1}  (  \mu_{1-y}-  \mu_y ) }{    \Vert \mu_1  -    \mu_0 \Vert_{ \Sigma^{-1}}  }      \right)      + 
 \\
		    \Phi \left(      \frac{    \left( X_\bullet - \mu_y\right)  ^\top  \Sigma^{-1}  (  \mu_{1-y}-  \mu_y  ) }{    \Vert \mu_1  -    \mu_0 \Vert_{ \Sigma^{-1}}  }  -  \Vert   \mu_1-  \mu_0 \Vert_{ \Sigma^{-1}}      \right)    - 1. \label{eq:PropGfX}
\end{multline}
Plugging~\eqref{eq:reecrits} and~\eqref{eq:PropGfX} into~\eqref{eq:transmdfz}, we finally then get
\begin{eqnarray*}
&& \mathbb{P}\left(      F_f^* \left( f(X_\bullet) \right)   \geq 1-\varepsilon              \ , \       s^*(X_\bullet)  \ne  Y_\bullet  \  | \  Y_\bullet = y \right)  
\\
& = &
 \mathbb{P}  \left(    \left\{   \Phi \left(Z_\bullet  \right)     +      \Phi \left( Z_\bullet  -   \Vert \mu_1  -    
 \mu_0 \Vert_{ \Sigma^{-1}}   \right)         \right\}  \geq 2 -  \varepsilon    \ , \  Z_\bullet   \geq  \frac{   
 \Vert \mu_1  -    \mu_0 \Vert_{ \Sigma^{-1}}   } {2}    \right),
\end{eqnarray*}
where $Z_\bullet \sim \mathcal{N}\left(0,1\right) $. The last equality is due to the fact that given $Y_\bullet =1$, the random variable $    \frac{    \left( X_\bullet - \mu_y\right)  ^\top  \Sigma^{-1}  (  \mu_{1-y}-  \mu_y  ) }{    \Vert \mu_1  -    \mu_0 \Vert_{ \Sigma^{-1}}  }  $ is normally distributed. 
We then get the desired result and the proof of the proposition is completed. 
\end{proof}

\subsubsection{Proposition~\ref{Th:RiskBayes}} 
Let $\eps \in ]0,1]$.
Since $ \mathbb{P}(Y_\bullet = 1  ) =  \mathbb{P}(Y_\bullet = 0  ) = 1/2 $, we have
\begin{multline*}
 \mathbb{P}( s^*(X_\bullet)  \ne Y_\bullet \ , \  F_f^* \left( f^*(X_\bullet) \right)  \geq 1-\varepsilon )
=
\frac{1}{2} \left\{  \mathbb{P}( s^*(X_\bullet)  \ne  Y_\bullet \ , \  F_f^* \left( f^*(X_\bullet) \right)   \geq 1-\varepsilon \ | \ Y_\bullet = 1 ) \right. 
\\
\left. + 
 \mathbb{P}( s^*(X_\bullet)   \ne Y_\bullet \ , \   F_f^* \left( f^*(X_\bullet) \right)  \geq 1-\varepsilon \ | \ Y_\bullet = 0) 
\right\}.
\end{multline*}
Next, using Proposition~\ref{prop:bayesProba}, we get
\begin{eqnarray*}
&& \mathbb{P}( s^*(X_\bullet)  \ne Y_\bullet \ , \  F_f^* \left( f^*(X_\bullet) \right)  \geq 1-\varepsilon )
\\
=  &&  
\mathbb{P}  \left(   \left\{   \Phi \left(Z   \right)     +      \Phi \left( Z  -   \Vert \mu_1  -    \mu_0 \Vert_{ \Sigma^{-1}}   \right)         \right\}  \geq 2 -  \varepsilon    \ , \  Z   \geq  \frac{   \Vert \mu_1  -    \mu_0 \Vert_{ \Sigma^{-1}}   } {2}  \right)
\\
= &&
\mathbb{P}\left(    \left\{   \Phi \left(Z   \right)     +      \Phi  \left( Z  -   \Vert \mu_1  -    \mu_0 \Vert_{ 
\Sigma^{-1}}   \right)         \right\}  \geq 2 -  \varepsilon   \right).
\end{eqnarray*}
The last equality is due to the following property: 
\begin{equation*}
Z < \frac{   \Vert \mu_1  -    \mu_0 \Vert_{ \Sigma^{-1}}   } {2}
\Rightarrow 
\Phi  \left(Z  -   \Vert \mu_1  -    \mu_0 \Vert_{ \Sigma^{-1}}    \right) < \Phi  \left( - \frac{ \Vert \mu_1  -    
\mu_0 \Vert_{ \Sigma^{-1}}}{2} \right) = 1 -  \Phi  \left(  \frac{ \Vert \mu_1  -    \mu_0 \Vert_{ \Sigma^{-1}}}{2} 
\right),
\end{equation*}
which implies that
\begin{equation*}
Z < \frac{ \Vert \mu_1  -    \mu_0 \Vert_{ \Sigma^{-1}}}{2}
\Rightarrow
\Phi \left(Z   \right)     +      \Phi  \left( Z  -   \Vert \mu_1  -    \mu_0 \Vert_{ \Sigma^{-1}}   \right) 
 < 1 \leq    2-\varepsilon.
\end{equation*}
The end of the proof is straightforward and follows from the relation $\Phi (x)+\Phi (-x) = 1,\, \forall x\in 
\mathbb{R}$. Indeed, we have
\begin{eqnarray*}
		\mathbb{P}\left(  \Phi \left(Z   \right)     +      \Phi  \left( Z  -   \Vert \mu_1  -    \mu_0 \Vert_{ \Sigma^{-1}}   \right)          
		\geq 2 -  \varepsilon \right)
& = &
		\mathbb{P}\left(  \Phi \left(- Z   \right)     +    \Phi  \left( - Z  +  \Vert \mu_1  -    \mu_0 \Vert_{ \Sigma^{-1}}   \right)       
		\leq  \varepsilon  \right)
\\ & = &
		\mathbb{P}\left( \Phi \left(Z   \right)     +    \Phi  \left( Z  +  \Vert \mu_1  -    \mu_0 \Vert_{ \Sigma^{-1}}   \right)    
		\leq  \varepsilon \right), 
\end{eqnarray*}
since $Z$ and $-Z $ equal in law. This ends the proof.


\subsection{Proof of Proposition~\ref{propo:presqconfset}}

We first define the following events
\begin{eqnarray*}
\mathcal{A}_y & = & \{f^*(\Xp) \geq \alpha_{\varepsilon},  \hat{f}(\Xp) < \hat{\alpha}_{\varepsilon}, s^{*}(\Xp) \neq y\},\;\; y= 0,1 \\
\mathcal{B}_y & = & \{f^*(\Xp) < \alpha_{\varepsilon},  \hat{f}(\Xp) \geq \hat{\alpha}_{\varepsilon},  \hat s(\Xp) \neq y\},\;\; y = 0,1.\\
\mathcal{C}_y & = &  \{f^*(\Xp) \geq \alpha_{\varepsilon}, \hat{f}(\Xp) \geq \hat{\alpha}_{\varepsilon}  ,  s^{*}(\Xp) \neq \hat s(\Xp), s^{*}(\Xp) \neq y\}, y = 0,1.
\end{eqnarray*}
Since $ \mathcal{P}(\widetilde{\Gamma}_\varepsilon^\bullet)  =  \varepsilon$, we can apply Proposition~2 and then, as 
\begin{equation*}
|2\eta^*(\Xp) - 1| \leq |\eta^*(\Xp) - \alpha_{\eps}| + |1-\eta^*(\Xp) - \alpha_{\eps}|,
\end{equation*} 
we deduce that
\begin{multline}
\label{eq:decomp0}
\mathbf{R} \left(\widetilde{\Gamma}_{\varepsilon}^{\bullet}\right) - \mathcal{R}\left(\Gamma_{\varepsilon}^{\bullet}\right) \leq \\
\dfrac{1}{\varepsilon}\{\mathbf{E}\left[|\eta^*(\Xp) - \alpha_{\eps}| \1_{\mathcal{A}_0 \cup \mathcal{B}_0 \cup \mathcal{C}_0 \cup \mathcal{C}_1}\right] + \mathbf{E}\left[|1 
- \eta^*(\Xp) - \alpha_{\eps}| \1_{\mathcal{A}_1 
\cup \mathcal{B}_1 \cup \mathcal{C}_0 \cup \mathcal{C}_1}\right]\}.
 \end{multline}
Now, 
\begin{enumerate}
\item on $\mathcal{A}_0$, $f^* = \eta^*$, $\eta^*(\Xp) \geq \alpha_{\eps}$ and $\hat{f}(\Xp) < \hat{\alpha}_{\eps}$,\\
hence, we have $|\eta^*(\Xp) - \alpha_{\eps}| \leq |\hat{\eta}(\Xp) - \eta^*(\Xp)|$ except if $\alpha_{\eps} \leq \hat{\alpha}_{\eps}$ and 
$\hat{f}(\Xp) \in (\alpha_{\eps}, \hat{\alpha}_{\eps})$;
\item on $\mathcal{B}_0$, $\hat{f} = \hat{\eta}$, $\hat{\eta}(\Xp) \geq \hat{\alpha}_{\eps}$ and $f^*(\Xp) < \alpha_{\eps}$,\\
hence, we have $|\eta^*(\Xp) - \alpha_{\eps}| \leq |\hat{\eta}(\Xp) - \eta^*(\Xp)|$ except if $\hat{\alpha}_{\eps} \leq \alpha_{\eps}$ and $\hat{f}(\Xp) \in 
(\hat{\alpha}_{\eps}, \alpha_{\eps})$;
\item on $\mathcal{C}_0$, $f^* = \eta^*$, $\hat{f} = 1-\hat{\eta}$, $\eta^*(\Xp) \geq \alpha_{\eps}$ and $\hat{\eta}(\Xp) \leq 1/2$,\\
hence, we always have $|\eta^*(\Xp) - \alpha_{\eps}| \leq |\hat{\eta}(\Xp) - \eta^*(\Xp)|$;
\item on $\mathcal{C}_1$, $f^* = 1-\eta^*$, $\hat{f} = \hat{\eta}$ and $\hat{\eta}(\Xp) \geq \hat{\alpha}_{\eps}$,\\
hence, we have $|\eta^*(\Xp) - \alpha_{\eps}| \leq |\hat{\eta}(\Xp) - \eta^*(\Xp)|$ except if $\hat{\alpha}_{\eps} \leq \alpha_{\eps}$ and $\hat{f}(\Xp) \in 
(\hat{\alpha}_{\eps}, \alpha_{\eps})$.
\end{enumerate}
Since $\mathcal{A}_0, \mathcal{B}_0, \mathcal{C}_0$ and $\mathcal{C}_1$ are mutually exclusive events, we deduce
\begin{multline}
\label{eq:decomp1}
\mathbf{E}\left[|\eta^*(\Xp) - \alpha_{\eps}| \1_{\mathcal{A}_0 \cup \mathcal{B}_0 \cup \mathcal{C}_0 \cup \mathcal{C}_1}\right] \leq
\mathbf{E}\left[|\eta^*(\Xp) - \alpha_{\eps}|\1_{\{|\hat{\eta}(\Xp) - \eta^*(\Xp)| \geq |\eta^*(\Xp) - \alpha_{\eps}|\}}\right] + \\
\mathbf{E}\left[|\eta^*(\Xp) - \alpha_{\eps}|
\left(\1_{\{\mathcal{A}_0,\alpha_{\eps} \leq \hat{\alpha}_{\eps}, \hat{f}(\Xp) \in (\alpha_{\eps}, \hat{\alpha}_{\eps})\}} 
+ \1_{\{\mathcal{B}_0 \cup \mathcal{C}_1, \hat{\alpha}_{\eps} \leq \alpha_{\eps}, \hat{f}(\Xp) \in (\hat{\alpha}_{\eps}, {\alpha}_{\eps})\}}\right) \right].
\end{multline}
%
In the same way, we obtain the following decomposition
\begin{multline}
\label{eq:decomp2}
\mathbf{E}\left[|1 - \eta^*(\Xp) - \alpha_{\eps}| \1_{\mathcal{A}_1 \cup \mathcal{B}_1 \cup \mathcal{C}_0 \cup \mathcal{C}_1}\right] \leq
\mathbf{E}\left[|1 - \eta^*(\Xp) - \alpha_{\eps}|\1_{\{|\hat{\eta}(\Xp) - \eta^*(\Xp)| \geq |1 - \eta^*(\Xp) - \alpha_{\eps}|\}}\right] + \\
\mathbf{E}\left[|1 - \eta^*(\Xp) - \alpha_{\eps}|
\left(\1_{\{\mathcal{A}_1,\alpha_{\eps} \leq \hat{\alpha}_{\eps}, \hat{f}(\Xp) \in (\alpha_{\eps}, \hat{\alpha}_{\eps})\}} 
+ \1_{\{\mathcal{B}_1 \cup \mathcal{C}_0, \hat{\alpha}_{\eps} \leq \alpha_{\eps}, \hat{f}(\Xp) \in (\hat{\alpha}_{\eps}, {\alpha}_{\eps})\}}\right) \right].
\end{multline}
Since $(\mathcal{A}_y,\mathcal{B}_y,\mathcal{C}_y), y = 0,1$ are mutually exclusive events, and that
 $|\eta^*(\Xp) - \alpha_{\eps}| \leq \alpha_{\eps}$ and $|1 - \eta(\Xp) - \alpha_{\eps}| \leq \alpha_{\eps}$, it derives from Inequalities~\eqref{eq:decomp0},
 \eqref{eq:decomp1} and~\eqref{eq:decomp2} that
\begin{multline*}
\mathbf{R} \left(\widetilde{\Gamma}_{\varepsilon}^{\bullet}\right) - \mathcal{R}\left(\Gamma_{\varepsilon}^{\bullet}\right) \leq \dfrac{1}{\varepsilon} \\
\{\mathbf{E}\left[ | \eta^*(\Xp) - \alpha_{\eps}|\1_{\{|\hat{\eta}(\Xp) - \eta^*(\Xp)| \geq |\eta^*(\Xp) - \alpha_{\eps}|\}}\right] + \\
\mathbf{E}\left[|1 - \eta^*(\Xp) - \alpha_{\eps}|\1_{\{|\hat{\eta}(\Xp) - \eta^*(\Xp)| \geq |1 - \eta^*(\Xp) - \alpha_{\eps}|\}}\right] + \\
\alpha_{\eps} |F_{\hat{f}}(\hat{\alpha}_{\eps}) - F_{\hat{f}}(\alpha_{\eps})|\}.
\end{multline*}
To conclude the proof, it remains to note that $1 - \eps  = F_{\hat{f}}(\hat{\alpha}_{\eps}) = F^*_{{f}}({\alpha}_{\eps})$, for all $\eps \in ]0,1]$.


\subsection{Proof of Theorem~\ref{thm:mainThm}}
We first set a Lemma that will be used in the proof.
\subsubsection{Tool lemma}
The following lemma is inspired by Lemma~3.1 in~\cite{AT07}.
\begin{lm}
\label{lem:lemUtile}
Let $X$ be a real random variable, $(X_n)_{n \geq 1}$ a be sequence of real random variables and $t_0 \in \mathbb{R}$. 
Assume that there exist $C_1 < \infty$ and $\gamma_0 > 0$ such that
\begin{equation*}
\mathbb{P}_{X} \left(|X - t_0 | \leq \delta  \right) \leq  C_1 \delta^{\gamma_0}, \qquad \forall \delta>0,
\end{equation*}
and a sequence of positive numbers $a_n \rightarrow +\infty$, $C_2, C_3$ some positive constants such that
\begin{equation*}
\mathbb{P}_{X_n}\left(|X_n - X| \geq \delta | X \right) \leq C_2 \exp\left(- C_3 a_n \delta^{2} \right), \qquad \forall \delta > 0,\ \forall n\in \mathbb{N}.
\end{equation*}
Then, there exists $C > 0$ depending only on $C_1,C_2$ and $C_3$, such that
\begin{eqnarray*}
\left|\mathbf{E}\left[\1_{\{X_n \geq t_0\}} - \1_{\{X \geq t_0\}} \right]\right| & \leq & \mathbf{E}\left[\left|\1_{\{X_n \geq t_0\}} - \1_{\{X \geq \alpha_{\eps}\}} \right| \right] \\ 
& \leq & \mathbf{P}\left(|X_n - X| \geq |X - t_0|\right) \\
& \leq & C a_n^{-\gamma_{0}/2}. 
\end{eqnarray*}
\end{lm} 

\begin{proof}
The following inequality holds
\begin{equation*}
\left| \1_{\{X_n \geq t_0\}} - \1_{\{X \geq t_0\}} \right| \leq \1_{\{|X_n - X| \geq |X-t_0|\}}.
\end{equation*}
Hence, it remains to prove
\begin{equation*}
 \mathbf{P}\left(|X_n - X| \geq |X-t_0|\right) \leq C a_n^{-\gamma_0/2}.
\end{equation*} 
We define, for $\delta > 0$,
\begin{eqnarray*}
\mathbf{A}_0 & = & \{\left|X-t_0\right| \leq \delta\}\\
\mathbf{A}_j & = & \{2^{j-1}\delta  < |X-t_0| \leq 2^{j} \delta\} , \;\; j \geq 1.
\end{eqnarray*}
Since the events $(\mathbf{A}_j)_{j\geq 0}$ are mutually exclusive, we deduce 
\begin{eqnarray*}
\mathbf{P}\left( |X_n - X| \geq |X-t_0|\right) & = & \sum_{j \geq 0} \mathbf{E}\left[\1_{\{|X_n - X| \geq |X-t_0|\}}\right] \1_{\mathbf{A}_j}\\
& \leq & \mathbb{P}_{X}\left(|X-t_0| \leq \delta \right) + \sum_{j \geq 1} \mathbf{E}\left[\1_{\{|X_n-X| \geq    2^{j-1}  \delta  \}} \1_{\mathbf{A}_j}\right]\\
& \leq & C_1 \delta^{\gamma_0}  + \sum_{j \geq 1} \mathbb{E}_{X} \left[\mathbb{P}_{X_n}\left(|X_n - X| \geq 2^{j-1}  \delta \  | X\right) \1_{\mathbf{A}_j} \right] \\
& \leq & C_1 \delta^{\gamma_0} + C_1C_2 \delta^{\gamma_0} \sum_{j \geq 1} 2^{j\gamma_0} \exp\left(-C_2 a_n 2^{2j-2} \delta^2 \right), 
\end{eqnarray*}
since $ \mathbb{P}_{X} (  \mathbf{A}_j ) \leq \mathbb{P}_{X}(|X-t_0| \leq 2^{j} \delta) \leq (2^j\delta)^{\gamma_0}$.
Therefore, choosing $\delta = a_n^{-1/2}$, we obtain from the above inequality,
\begin{eqnarray*}
\mathbf{P}\left(|X_n - X| \geq |X-t|\right)  & \leq & C_1 a_n^{-\gamma_0 / 2} +2  C_1C_3 a_n^{-\gamma_0 / 2} \sum_{j \geq 1} 2^{j\gamma_0} \exp(-C_2 2^{2j-2}) \\
& \leq & C a_n^{-\gamma_0 /2}, 
\end{eqnarray*}
for a constant $C > 0$.
\end{proof}


\subsubsection{Theorem~\ref{thm:mainThm}}
Let $\eps \in ]0,1[$
We first prove that for $N$ large enough
\begin{equation}
\label{finaleq}
\left|\mathbf{P}\left(\hat{F}_{\hat{f}}(\hat{f}(\Xp)) \geq 1 - \eps\right) - \mathbf{P}\left({F}_{\hat{f}}(\hat{f}(\Xp)) \geq 1 - \eps\right)\right|
\leq \dfrac{\tilde{C}}{\sqrt{N}}, 
\end{equation}
and
\begin{equation}
\label{eq:firstStep}
\left|\mathbf{R} \left(\widehat{\Gamma}_{\varepsilon}^{\bullet}\right) - \mathbf{R} \left(\widetilde{\Gamma}_{\varepsilon}^{\bullet}\right)\right|
 \leq \dfrac{C}{\sqrt{N}}, 
\end{equation}
where $C, \tilde{C} > 0$ are constants which do not depend on $n$.
For all $x \in[1/2,1]$, 
\begin{equation*}
F_{\hat{f}}(x) = \mathbb{E}_{\mathcal{D}_n}\left[\mathbb{P}\left(\hat{f}(X) \leq x | \mathcal{D}_{n}\right)\right],
\end{equation*}
Hence, conditional on $\mathcal{D}_n$, $\hat{F}_{\hat{f}}(x)$ is the empirical cumulative distribution function
of $\hat{f}(X)$, where $\hat{f}$ is view as a deterministic function.
Therefore, for all $\gamma \geq \sqrt{\log(2)/2N}$, Dvoretsky-Kiefer-Wolfowitz Inequality yields
\begin{eqnarray*}
\mathbb{P}_{\mathcal{D}_N}\left(|\hat{F}_{\hat{f}}(\hat{f}(\Xp)) - F_{\hat{f}}(\hat{f}(\Xp))| \geq \gamma | \mathcal{D}_n, \Xp \right)  & \leq &
\mathbb{P}_{\mathcal{D}_N}\left(\sup_{x \in [1/2,1]} \left|\hat{F}_{\hat{f}}(x)  - U_n(x) \right| \geq \gamma | \mathcal{D}_n \right)\\
& \leq & 2 \exp(-2N \gamma^2),
\end{eqnarray*}
where $U_n(x) = \mathbb{P}\left(\hat{f}(\Xp) \leq x | \mathcal{D}_n\right)$.\\
Applying Lemma~\ref{lem:lemUtile}, we get
\begin{equation*}
\left|\mathbb{E}_{(\mathcal{D}_N,\Xp)}\left[\1_{\{\hat{F}_{\hat{f}}(\hat{f}(\Xp)) \geq 1 - \eps\}} -
 \1_{\{{F}_{\hat{f}}(\hat{f}(\Xp)) \geq 1 - \eps\}}|\mathcal{D}_n \right] \right|
  \leq  \dfrac{\tilde{C}}{\sqrt{N}},
\end{equation*}
where $\tilde{C}$ does not depend on $n$. Hence, we obtain Inequality~\eqref{finaleq}.
In the same way, we have
\begin{equation*}
\left|\mathbf{P}\left(\hat{F}_{\hat{f}}(\hat{f}(\Xp)) \geq 1 - \eps, \hat{s}(\Xp) \neq \Yp\right) - \mathbf{P}\left({F}_{\hat{f}}(\hat{f}(\Xp)) \geq 1 - \eps, \hat{s}(\Xp) 
\neq \Yp \right)\right|
\leq \dfrac{\tilde{C}}{\sqrt{N}}.
\end{equation*}

Therefore, Inequality~\eqref{eq:firstStep} holds for some constant $C>0$.

Since, by Assumption (A2) $\mathbf{P}\left({F}_{\hat{f}}(\hat{f}(\Xp)) \geq 1 - \eps\right) = \varepsilon$, Inequality~\eqref{finaleq} yields
\begin{equation*}
\mathbf{P} \left(\hat{F}_{\hat{f}}(\hat{f}(\Xp))  \geq 1- 
		\varepsilon \right)
		=    \eps + O(N^{-1/2}).
\end{equation*}

Now, we conclude the point $1)$ of the theorem. Since Inequality~\eqref{eq:firstStep} ensures that 
\begin{equation*}
\left|\mathbf{R} \left(\widehat{\Gamma}_{\varepsilon}^{\bullet}\right) - \mathbf{R} \left(\widetilde{\Gamma}_{\varepsilon}^{\bullet}\right)\right| \rightarrow 0,
\;\; n,N \to +\infty,
\end{equation*}
it remains to prove that 
\begin{equation*}
\left|\mathbf{R} \left(\widetilde{\Gamma}_{\varepsilon}^{\bullet}\right) - \mathcal{R} \left({\Gamma}_{\varepsilon}^{\bullet}\right)\right| \rightarrow 0,
\;\; n \to +\infty.
\end{equation*}
Applying Proposition~3, we obtain for $\delta_n >0, \;\; \delta_n \to 0$
\begin{equation*}
\left|\mathbf{R} \left(\widetilde{\Gamma}_{\varepsilon}^{\bullet}\right) - \mathcal{R} \left({\Gamma}_{\varepsilon}^{\bullet}\right)\right| \leq
2\delta_n + 2\mathbf{P}\left(|\hat{\eta}(\Xp) - \eta^*(\Xp)| \geq \delta_n \right) + \left|F_{\hat{f}}(\alpha_{\varepsilon}) - F_{f}^*(\alpha_{\varepsilon}) \right|.
\end{equation*}
Since, $\hat{\eta}(\Xp) \rightarrow \eta^*(\Xp)$ in probability when $n \rightarrow +\infty$,
$\hat{f}(\Xp) \to f^*(\Xp)$ in distribution and $\left|F_{\hat{f}}(\alpha_{\varepsilon}) - F_{f}^*(\alpha_{\varepsilon}) \right| \to 0$.
Moreover, $\mathbf{P}\left(|\hat{\eta}(\Xp) - \eta^*(\Xp)| \geq \delta_n \right) \to 0$ which concludes the point $1)$ of the proof.

Finally, to prove $2)$, it remains to show that 
\begin{equation*}
\left|\mathbf{R} \left(\widetilde{\Gamma}_{\varepsilon}^{\bullet}\right) - \mathcal{R} \left({\Gamma}_{\varepsilon}^{\bullet}\right)\right| = O(a_n^{-\gamma_{\eps}/2}), 
\end{equation*}
We first note that,
\begin{eqnarray*}
\left|{F}_{\hat{f}}(\alpha_{\eps}) - F_f^*(\alpha_{\eps}) \right| & \leq & \mathbf{E} \left[|\1_{\{\hat{f}(\Xp) \geq \alpha_{\eps}\}} -\1_{\{{f}^*(\Xp) \geq \alpha_{\eps}\}}| 
\right]\\
& \leq & \mathbf{E} \left[|\1_{\{\hat{f}(\Xp) \geq \alpha_{\eps}\}} -\1_{\{{f}^*(\Xp) \geq \alpha_{\eps}\}}|\1_{\{|\hat{\eta}(\Xp) - \eta^*(\Xp)| \geq |\eta^*(\Xp) - 1/2| \}} 
\right]\\
&  & \;\; + \mathbf{E} \left[|\1_{\{\hat{f}(\Xp) \geq \alpha_{\eps}\}} -\1_{\{{f}^*(\Xp) \geq \alpha_{\eps}\}}|\1_{\{|\hat{\eta}(\Xp) - \eta^*(\Xp)| < |\eta^*(\Xp) - 1/2| \}} 
\right]\\
& \leq & \mathbf{P}\left(|\hat{\eta}(\Xp) - \eta^*(\Xp)| \geq |\eta^*(\Xp) - \alpha_{\eps}| \right) \\
&  & \;\; + \mathbf{P}\left(|\hat{\eta}(\Xp) - \eta^*(\Xp)| \geq |\eta^*(\Xp) - (1 - \alpha_{\eps})| \right) \\
&  & \;\;\;\;\; + \mathbf{E} \left[|\1_{\{\hat{\eta}(\Xp) \geq \alpha_{\eps}\}} -\1_{\{\eta^*(\Xp) \geq \alpha_{\eps}\}}| \right].\\
\end{eqnarray*}
Therefore, applying both Proposition~3, Lemma~\ref{lem:lemUtile} and the above inequality, we get the desired result.


\bibliographystyle{alpha}      
\bibliography{Conformal_lasso3}

\end{document}